\documentclass[12pt]{amsart}
\usepackage{amssymb, eucal, amsfonts, amsmath, xy,latexsym, bbm}
\usepackage{mathrsfs}

\usepackage{xcolor}
\usepackage[margin=2cm]{geometry}
\usepackage[colorlinks,hypertexnames=true,linkcolor=teal,urlcolor=black,citecolor=black]{hyperref}

\numberwithin{equation}{section}

\newtheorem{Theorem}{Theorem}[section]
\newtheorem*{Theorem*}{Theorem}
\newtheorem*{Corollary*}{Corollary}

\newtheorem{Lemma}[Theorem]{Lemma}

\theoremstyle{definition}

\theoremstyle{remark}
\newtheorem{Remark}[Theorem]{Remark}
\newtheorem*{Remark*}{Remark}

\newcommand{\C}{\mathbb{C}}

\newcommand{\Z}{\mathbb{Z}}

\renewcommand{\k}{\mathbbm{k}}
\renewcommand{\Z}{\mathbb{Z}}

\newcommand{\g}{\mathfrak{g}}
\newcommand{\p}{\mathfrak{p}}

\renewcommand{\t}{\mathfrak{t}}

\renewcommand{\sl}{\mathfrak{sl}}

\newcommand{\n}{\mathfrak{n}}

\renewcommand{\O}{\mathcal{O}}

\newcommand{\I}{\mathcal{I}}

\newcommand{\V}{\mathcal{V}}

\renewcommand{\P}{{\mathcal{P}}}

\newcommand{\F}{\mathcal{F}}

\newcommand{\ad}{\operatorname{ad}}
\newcommand{\Ad}{\operatorname{Ad}}

\newcommand{\Lie}{\operatorname{Lie}}

\newcommand{\gr}{\operatorname{gr}}

\newcommand{\End}{\operatorname{End}}

\renewcommand{\i}{{\bf i}}
\renewcommand{\j}{{\bf j}}

\begin{document}
\title[The number of multiplicity-free primitive ideals]{The number of multiplicity-free primitive ideals
associated with the rigid nilpotent orbits}

\author{Alexander Premet and David~I.~Stewart}

\thanks{\nonumber{\it Mathematics Subject Classification} (2000 {\it revision}).
Primary 17B35, 17B50. Secondary 17B20.}
\address{Department of Mathematics, The University of Manchester, Oxford Road, M13 9PL, UK} 
\email{alexander.premet@manchester.ac.uk}
\address{Department of Mathematics, The University of Manchester, Oxford Road, M13 9PL, UK} 
\email{david.i.stewart@manchester.ac.uk}
 \maketitle
\begin{center}
	{\it To Corrado De Concini with admiration}
\end{center}

\begin{abstract}
	\noindent
	Let $G$ be a simple algebraic group defined over $\C$ and let $e$ be a rigid nilpotent element
	in $\g =\Lie(G)$. In this paper we prove that the finite $W$-algebra $U(\g,e)$ admits either one or two $1$-dimensional representations. Thanks to the results obtained earlier this boils down to showing that the finite $W$-algebras associated with the rigid nilpotent orbits of dimension $202$ in the Lie algebras of type ${\rm E}_8$ admit exactly two $1$-dimensional representations.
	 As a corollary, we complete the description of the multiplicity-free primitive ideals 
	of $U(\g)$ associated with the rigid nilpotent $G$-orbits of $\g$. At the end of the paper, we apply our results to enumerate the small irreducible representations of the related reduced enveloping algebras.
	\end{abstract}
\maketitle

\section{Introduction}

Denote by $G$ a simple algebraic group of adjoint type over
$\C$ with Lie algebra $\g=\Lie(G)$ and let  $\mathcal{X}$ be the set of all primitive ideals of the universal enveloping algebra $U(\g)$. We shall identify $\g$ and $\g^*$ by means of an $(\Ad G)$-invariant non-degenerate symmetric bilinear form $(\,\cdot\,,\,\cdot\,)$ of $\g$. 
Given $x\in \g$ we write $G_x$ the centraliser of $x$ in $G$ and write $\g_x:=\Lie(G_x)$.

It is well known that for any finitely generated $S(\g^*)$-module $M$ there exist
prime ideals $\mathfrak{q}_1,\ldots,\mathfrak{q}_n$ containing ${\rm Ann}_{\,S(\g^*)}\,M$ and a chain
$0=R_0\subset R_1\subset\ldots
\subset R_n=R$ of $S(\g^*)$-modules such that $R_i/R_{i-1}\cong S(\g^*)/{\mathfrak q}_i$ for $1\le i\le n$.
Let $\p_1,\ldots,\p_l$ be the minimal elements in the set $\{\mathfrak{q}_1,\ldots,\mathfrak{q}_n\}$. The zero sets $\V(\p_i)$ of the $\p_i$'s in $\g$ are the irreducible components of the support ${\rm Supp}(M)$ of $M$. If $\p$ is one of the $\p_i$'s then we define $m(\p):=\{1\le i\le n\,|\,\,\mathfrak{q}_i=\p\}$ and we call $m(\p)$ the {\it multiplicity} of $\V(\p)$ in ${\rm Supp}(M)$. The formal linear combination $\sum_{i=1}^lm(\p_i)[\p_i]$ is  often referred to as the {\it associated cycle} of $M$ and denoted ${\rm AC}(M)$.

Given $I\in\mathcal{X}$ we can apply the above construction to the $S(\g^*)$-module $S(\g^*)/{\rm gr}(I)$ where $\gr(I)$ is the corresponding graded ideal in
$\gr(U(\g))=S(\g)\cong S(\g^*)$.
The support of $S(\g^*)/\gr(I)$ in $\g$ is called the {\it associated variety} of $I$ and denoted ${\rm V}(I)$. By Joseph's theorem, ${\rm V}(I)$ is the closure of a single nilpotent orbit $\mathcal O$ in $\g$ and, in particular, it is always irreducible. Hence in our situation the set $\{\p_1,\ldots,\p_l\}$ is  the singleton containing $J:=\sqrt{\gr(I)}$ and we have that ${\rm AC}\big(S(\g^*)/\gr(I)\big)=m(J)[J]$. The positive integer
$m(J)$ is referred to the {\it multiplicity} of $\mathcal O$ in $U(\g)/I$ and denoted ${\rm mult}_{\mathcal O}(U(\g)/I)$.  

For a nilpotent orbit $\mathcal O$ in $\g$ we denote by $\mathcal{X}_{\mathcal O}$ the set of all $I\in\mathcal{X}$ with ${\rm V}(I)=\overline{\mathcal O}$. Following \cite{Pr14} we call $I\in\mathcal{X}_{\mathcal O}$ {\it multiplicity-free} if ${\rm mult}_{\mathcal O}(U(\g)/I)=1$ and we say that a $2$-sided ideal $J$ of $U(\g)$ is {\it completely prime} if $U(\g)/J$ is a domain.  

Classification of completely prime primitive ideals of $U(\g)$ is a classical problem of Lie Theory
which finds applications in the theory of unitary representations of complex simple Lie groups. The subject has a very long history and many partial results can be found in the literature. In particular, it is known that any multiplicity-free primitive ideal is completely prime and that the converse fails outside type $\sf A$ for simple Lie algebras of rank $\ge 3$; see \cite{Pr11} and \cite{LP} for more detail.
A description of multiplicity-free primitive ideals in Lie algebras of types $\rm B$, {\rm C} and $\rm D$ was first obtained in \cite{PT14}; that paper also solved the problem fo the majority of induced nilpotent orbits in exceptional Lie algebras.

Fix a nonzero nilpotent orbit $\mathcal O\subset \g$ and let $\{e,h, f\}$ be an $\sl_2$-triple in $\g$ with $e\in{\mathcal O}$. Let $Q$ be the generalised Gelfand--Graev module associated
with $\{e,h, f\}$; see \cite{PT21} for more detail. Let
$U(\g,e):=({\rm End}_\g\, Q)^{\rm op}$, the finite $W$-algebra associated with $(\g,e)$.
If $V$ is a finite dimensional irreducible
$U(\g,e)$-module, then Skryabin's theorem \cite[Appendix]{Pr02} in conjunction with \cite[Theorem~3.1(ii)]{Pr07a} implies $Q\otimes_{U(\g,\,e)}V$ is an irreducible $\g$-module
and its annihilator $I_V$ in $U(\g)$ lies in $\mathcal{X}_{\mathcal O}$. Conversely, any primitive ideal in $\mathcal{X}_{\mathcal O}$ has this form for some finite dimensional
irreducible $U(\g,e)$-module $V$. This result was conjectured in \cite[3.4]{Pr07a} and proved in \cite[Theorem~1.1]{Pr07b} for the primitive ideals admitting rational central characters. In full generality, the conjecture was first proved by Losev; see \cite[Theorem~1.2.2(viii)]{Lo2}. A bit later, alternative proofs were found by Ginzburg in \cite[4.5]{gi} and by the first-named author in \cite[Sect.~4]{Pr10}.
The ideal $I_V$ depends only on the image of $V$ in the set ${\rm Irr}\,U(\g,e)$ of all isoclasses of finite dimensional irreducible $U(\g,e)$-modules. We write $[V]$ for the class of $V$ in ${\rm Irr}\,U(\g,e)$.

It is well-known that
group $C(e):=G_e\cap G_f$  is reductive and its finite quotient
$\Gamma:=C(e)/C(e)^\circ$ identifies with the component group of
the centraliser $G_e$. From the Gan--Ginzburg realization of the finite $W$-algebra $U(\g,e)$ it follows that $C(e)$ acts
on $U(\g,e)$ by algebra automorphisms; see \cite[Theorem~4.1]{GG02}. 
By \cite[Lemma~2.4]{Pr07a}, the connected component
$C(e)^\circ$ preserves any $2$-sided ideal of $U(\g,e)$. As a result,  we have a natural action of $\Gamma$ on ${\rm Irr}\,U(\g,e)$. For $V$ as above, we let $\Gamma_V$ denote the stabiliser of $[V]$ in $\Gamma$. In \cite[4.2]{Lo5}, Losev proved that
$I_{V'}=I_V$ if an only if $[V']=[V]^\gamma$ for some $\gamma\in\Gamma$. In particular, $\dim V=\dim V'$. 
In conjunction with \cite[Theorem~1.3.1(2)]{Lo5}, this result of Losev also implies that
$${\rm mult}_{\mathcal O}(U(\g)/I_V)\,=\,[\Gamma:\Gamma_V]\cdot (\dim V)^2.$$ As a consequence,
a primitive ideal $I_V$ is multiplicity-free if and only if $\dim V=1$ and $\Gamma_V=\Gamma$.
This brings our attention to the set $\mathcal{E}$ of all one-dimensional representations of $U(\g,e)$ and its subset $\mathcal{E}^\Gamma$ consisting of all $C(e)$-stable such representations. 
Since
$\mathcal{E}$ identifies with the maximal spectrum of the largest commutative quotient $U(\g,e)^{\rm ab}$ of $U(\g,e)$, it follows that $\mathcal{E}$ is an affine variety and $\mathcal{E}^\Gamma$ is a Zariski closed subset of $\mathcal{E}$.

If $\g$ is a classical Lie algebra then it is proved in \cite[Theorem~1]{PT14} the variety $\mathcal{E}^\Gamma$ is isomorphic to the affine space $\mathbb{A}^{c_\Gamma(e)}$ where 
$c_\Gamma(e)=\dim (\g_e/[\g_e,\g_e])^\Gamma$ (one should keep in mind here that the connected component of $G_e$ acts trivially on
$\g_e/[\g_e,\g_e]$). This result continues to hold for $\g$ exceptional provided that the orbit $\O$ is induced (in the sense of Lusztig--Spaltenstein) and not listed in \cite[Table~0]{PT14}. That table contains seven induced orbits (one in types ${\rm F}_4$, ${\rm E}_6$, ${\rm E}_7$ and four in type ${\rm E}_8$).

It is also known that $\mathcal{E}\ne\varnothing$ for all nilpotent orbits $\O$ in the finite dimensional simple Lie algebras $\g$ and $\mathcal{E}$ is a finite set if and only if the orbit $\O\subset \g$
is {\it rigid}, that is cannot be induced from a proper Levi subalgebra of $\g$ in the sense of Lusztig--Spaltenstein.
 This was first conjectured in \cite[Conjecture~3.1]{Pr07a}.  Several mathematicians contributed to the proof of this conjecture and we refer to \cite[Introduction]{Pr14} for more detail on the history of the subject.

Furthermore, it is known that $\mathcal{E}^\Gamma\ne \varnothing$ in all cases.
If $e$ is rigid and $\g$ is classical then $\g_e=[\g_e,\g_e]$
by \cite{Ya10}, whilst if $\g$ is exceptional then 
either $\g_e=[\g_e,\g_e]$ or $\g_e=\mathbb{C}e\oplus[\g_e,\g_e]$ and the second case occurs for one rigid orbit in types ${\rm G}_2$, ${\rm F}_4$, ${\rm E}_7$ and for three rigid orbits in type ${\rm E}_8$; see \cite{dG, PS18}. The Bala--Carter labels of these orbits are listed in Table~\ref{t:orbittable}.

\begin{table}[htb]
	\bgroup
	\def\arraystretch{1.5}
	\begin{tabular}{| c | | c | c | c | c | c | c |} 
		\hline
		{\rm Type of} $\Phi$ &  ${\rm G}_2$ &  ${\rm F}_4$ & ${\rm E}_7$ & ${\rm E}_8$ & ${\rm E}_8$ & ${\rm E}_8$ \\
		\hline 
		{\rm Bala--Carter label of} $e$ & $\widetilde{\rm A}_1$ & $\widetilde{\rm A}_2 + {\rm A}_1$ & $({\rm A}_3 + {\rm A}_1)' $ & ${\rm A}_3 + {\rm A}_1$ & ${\rm A}_5 + {\rm A}_1$  & ${\rm D}_5({\rm a}_1) + {\rm A}_2$\\
		\hline 
	\end{tabular}
	\egroup\vspace{6pt}
	\caption{Rigid nilpotent elements with imperfect centralisers.}
	\label{t:orbittable}
	\end{table}
\noindent
Since $\mathcal{E}^\Gamma\ne\varnothing$, it follows from
\cite[Proposition~11]{PT14} that for any simple Lie algebra $\g$ the equality $\g_e=[\g_e,\g_e]$
implies that $\mathcal{E}$ is a singleton. In view of the above we see that for any rigid nilpotent element in a classical Lie algebra the set
$\mathcal{E}=\mathcal{E}^\Gamma$ contains one element, whilst for $\g$ exceptional and $e$ rigid the inequality $|\mathcal{E}|\ge 2$ may occur only for the six orbits listed in Table~\ref{t:orbittable}. 

Let $T$ be a maximal torus of $G$ and $\t=\Lie(T)$. Let $\Phi$ be the root system of $\g$ with respect to $T$ and let $\Pi$ be a basis of simple roots in $\Phi$.  By Duflo's theorem \cite{Du}, any primitive ideal $I\in\mathcal{X}$  has the form 
$I=I(\lambda):={\rm Ann}_{U(\g)}\,L(\lambda)$ for some irreducible highest weight $\g$-modules $L(\lambda)$ with $\lambda\in\t^*$, and all multiplicity-free primitive ideals $I$ constructed in \cite{Pr14} are given in their Duflo realisations.
It is known that if 
$\langle\lambda,\alpha^\vee\rangle\in\Z$ for all $\alpha\in\Pi$ then ${\rm V}(I)$ is the closure of a {\it special} (in the sense of Lusztig) nilpotent orbit in $\g$. One also knows that to any $\sl_2$-triple $\{e,h,f\}$ in $\g$ with $e$ special there corresponds an $\sl_2$-triple $\{e^\vee,h^\vee,f^\vee\}$ in the Langlands dual Lie algebra $\g^\vee$ with $h^\vee\in\t^*$.  As Barbasch--Vogan observed in \cite[Proposition~5.10]{BV}, for $e$  special and rigid there is a unique choice of $h^\vee$ such that
$\langle\frac{1}{2}h^\vee,\alpha^\vee\rangle\in\{0,1\}$ for all $\alpha\in\Pi$. Furthermore, in this case we have that $I(\frac{1}{2}h^\vee-\rho)\in\mathcal{X}_\O$ (here $\rho$ is the half-sum of the positive roots of $\Phi$ with respect to $\Pi$ and $\mathcal O$ is the nilpotent orbit containing $e$).

If $\g$ is classical and $e$ is special rigid, then it follows from \cite{McG} that one of the Duflo realisations of the multiplicity-free primitive ideal in $\mathcal{X}_{\mathcal O}$ is obtained by using
the Arthur--Barbasch--Vogan recipe described above. By \cite[Theorem~A]{Pr14}, this result continues to hold for the special rigid nilpotent orbits in exceptional Lie algebras. (It is worth mentioning here that all nilpotent elements listed in Table\ref{t:orbittable} are non-special.) It was also proved in \cite{Pr14} that for any orbit $\O$ listed in Table~\ref{t:orbittable} the set 
$\mathcal{X}_\O$ contains (at least) two multiplicity-free primitive ideals and their Duflo realisations 
$I(\Lambda)$ and $I(\Lambda')$ were found in all cases by using a method described by Losev in \cite[5.3]{Lo3}.

It should be stressed at this point that in the case of rigid nilpotent orbits in exceptional Lie algebras the set $\mathcal{E}$ was first investigated 
by Goodwin--R{\"o}hrle--Ubly \cite{GRU} and Ubly \cite{U} who relied on some custom GAP code.
In particular, it was checked in \cite{GRU} that $|\mathcal{E}|=2$ for all 
orbits in types ${\rm G}_2$, ${\rm F}_4$ and ${\rm E}_7$ listed in Table~\ref{t:orbittable}. After \cite{GRU} was submitted Ubly has improved the GAP code and was able to check that $|\mathcal{E}|=2$ for the nilpotent orbit in type with Bala--Carter label ${\rm A}_3+{\rm A}_1$ in type  ${\rm E}_8$; see
\cite{U}. This left open the two largest rigid nilpotent orbits (of dimension $202$) in Lie algebras of type ${\rm E}_8$.

The main result of this paper is the following:

\medskip

\noindent {\bf Theorem~A.} {\it If $e$ lies in a nilpotent orbit $\O$ listed in Table~\ref{t:orbittable} then
$|\mathcal{E}|=|\mathcal{E}^\Gamma|=2$. Consequently, the set $\mathcal{X}_\O$ contains two
multiplicity-free primitive ideals.}
	
\medskip

\noindent 
Combined with the main results of \cite{Pr14}, Theorem~A provides a full list of 
all multiplicity-free primitive ideals of $U(\g)$ associated with rigid nilpotent orbits. 
Since $\Gamma=\{1\}$ for all nilpotent elements listed in Table~\ref{t:orbittable}, in order to prove the theorem we just need to show that $|\mathcal{E}|=2$ for the nilpotent elements in Lie algebras of type $\rm{E}_8$ labelled
${\rm A}_5 + {\rm A}_1$  and ${\rm D}_5({\rm a}_1) + {\rm A}_2$. By the proof of Proposition~2.1 in \cite{Pr14} and by \cite[Proposition~5.4]{PT21}, the largest commutative quotient $U(\g,e)^{\rm ab}$ of $U(\g,e)$ is generated by the image of a Casimir element of $U(\g)$ in $U(\g,e)^{\rm ab}$; we call it $c$. Looking very closely at the commutators of certain PBW generators of Kazhdan degree $5$ in $U(\g,e)$ we are able to  show that $\lambda c^2+\eta c +\xi=0$ for some $\lambda\in \C^\times$  and $\eta,\xi\in\C$. This quadratic equation results from investigating certain elements of Kazhdan degree $8$ in the graded Poisson algebra $\P(\g,e)$ associated with the Kazhdan filtration of $U(\g,e)$.

Let $R=\Z[\frac{1}{2},\frac{1}{3},\frac{1}{5}]$. In \cite[4.1]{PT21},  a natural  $R$-form, $Q_R$, of the Gelfand--Graev  module $Q$ was introduced, and it was proved for $e$ rigid that the ring
$U(\g_R,e):=\End_\g(Q_R)^{\rm op}$ has a nice PBW basis over $R$. In the present paper, we use these results to carry out all our computations over the ring $R$. In particular, we show that $\lambda\in R^\times$ and $\eta,\xi\in R$. The explicit form of $\Lambda$ and $\Lambda'$ in \cite[3.16, 3.17]{Pr14} in conjunction with
\cite[Theorem~1.2]{PT21} and \cite[Theorem~2.3]{Pr02} then enables us to obtain the following:

\medskip

\noindent {\bf Theorem~B.} {\it Let $\g_\k=\Lie(G_\k)$ be a Lie algebra of type ${\rm E}_8$ over an algebraically closed field $\k$ of charateristic $p>5$ and let $e$ be a nilpotent element of $\g_\k$ with Bala--Carter label ${\rm A}_5+{\rm A}_1$ or ${\rm D}_5({\rm a}_1)+{\rm A}_2$.
Let $\chi\in \g_\k^*$ be such that
$\chi(x)=\kappa(e,x)$ for all $x\in\g_\k^*$ where $\kappa$ is the Killing form of $\g_\k$.
Then the reduced enveloping algebra $U_\chi(\g_\k)$ has two simple modules of dimension
$p^{d(\chi)}$ where $d(\chi)=101$ is half the dimension of the coadjoint $G_\k$-orbit of $\chi$.}

\medskip

\noindent
We recall that $U_\chi(\g_\k)\,=\,U(\g_\k)/I_\chi$ where $I_\chi$ is the $2$-sided ideal of $U(\g_\k)$ generated by all elements $x^p-x^{[p]}-\chi(x)^p$ with $x\in\g_\k$
(here $x\mapsto x^{[p]}$ is the $[p]$-th power map of the restricted Lie algebra $\g_\k$). 
By the Kac--Weisfeiler conjecture (proved in \cite{Pr95}) any finite-dimensional $U_\chi(\g_\k)$-module has dimension divisible by $p^{d(\chi)}$. 
It would be interesting to prove an analogue of Theorem~B for the first four orbits in Table~\ref{t:orbittable} and to reestablish the remaining results of \cite{GRU} and \cite{U}
by the methods of the present paper.

\bigskip
\noindent {\bf Acknowledgement.} 
The first-named author would like to thank Shilin Yu and Lucas Mason-Brown for asking stimulating questions related to \cite{Pr14}. Thanks also go to
Ross Lawther and Donna Testerman for helpful e-mail
correspondence on the subject of this paper.

\section{Notation and preliminaries}
\label{basic}
Let $G_\Z$ be a Chevalley group scheme of type ${\rm E}_8$ and $\g_\Z=\Lie(G_\Z)$. Let $R=\Z[\frac{1}{2},\frac{1}{3},\frac{1}{5}]$ (recall that $2$, $3$ and $5$ are bad primes for $G_\Z$). 
We set $\g_R:=\g_\Z\otimes_\Z R$, and $\g:=\g_\Z\otimes_\Z\C$.
Let $\Phi$ be the root system
of $G_\Z$ with respect to a maximal split torus $T_\Z$ of $G_\Z$. Let $\Pi=\{\alpha_1,\ldots,\alpha_8\}$ be a set of simple roots in $\Phi$ and write $\Phi_+$ for the set of positive roots of $\Phi$ with respect to $\Phi$. We always use Bourbaki's numbering of simple roots; see \cite[Planche~VII]{B}.

We choose a Chevalley system $\bigcup_{\alpha\in\Phi_+}\{h_\alpha,e_\alpha,f_\alpha\}$ of $\g_\Z$ so that that the signs of the structure constants $N_{\alpha,\beta}\in\{-1,0,1\}$ with $\alpha,\beta\in\Phi$ follow the conventions of \cite{GS} and \cite{LS}. Recall that $h_\alpha=[e_\alpha,f_\alpha]$ for all $\alpha\in\Phi_+$. We set $e_i:=e_{\alpha_i}$, 
$f_i:=f_{\alpha_i}$ and $h_i:=h_{\alpha_i}$ for all $\alpha_i\in\Pi$ and denote by $(\,\cdot\,,\,\cdot\,)$ the $\Z$-valued invariant symmetric bilinear form on $\g_\Z$ such that $(e_\alpha,f_\alpha)=1$ for all $\alpha\in\Phi_+$. 

Given $x\in \g$ we denote by $\g_x$ the centraliser of $x$ in $\g$. Of course, our main concern is with the nilpotent elements $e\in\g_\Z$ labelled ${\rm A}_5 + {\rm A}_1$  and ${\rm D}_5({\rm a}_1) + {\rm A}_2$. A lot of useful information on the structure of $\g_e$ can be found in \cite[pp.~149, 150]{LT2}. We note that the cocharacter $\tau\in X_*(T_\Z)$ introduced in {\it op.\,cit.} 
is optimal for $e$ in the sense of the Kempf--Rousseau theory; see \cite{Pr03} for detail. The adjoint action of $\tau(\C^*)$ on $\g$ gives rise to a $\Z$-grading $\g_e=\bigoplus_{i\in\Z_{\ge 0}}\g_e(e)$ of $\g_e$. As explained in \cite[3.4]{PT21}, this grading is defined over $R$, that is $\g_{R,e}:=\g_e\cap \g_R=\bigoplus_{i\in\Z_{\ge 0}}\g_{R,e}(i)$ where $\g_{R,e}(i)=\g_R\cap\g_e(i)$. Also, $\g_{R,e}$ is a direct summand of the Lie ring $\g_R$.  

In what follows we adopt the notation introduced in \cite{Pr02} and \cite{PT21}. 
Let $Q$ be the generalised Gelfand--Graev module associated with $e$ and write $Q_R$ and $U(\g_R,e)$ for the $R$-forms of $Q$ and $U(\g,e)$ defined in \cite[4.1, 5.1]{PT21}. We write $\F_i(Q)$ and $\F_i(Q_R)$ for the $i$-th components of the Kazhdan filtration of $Q$ and $Q_R$, respectively, and regard $U(\g,e)$ as a subspace of $Q$. 
By \cite[4.5]{PS18} and \cite[Sect.~5]{PT21}, the associative algebra $U(\g,e)$ is generated by elements $\Theta_y$ with $y\in \bigcup_{i\le 5}\,\g_e(i)$ and every such element is defined over $R$, i.e. has the property  \begin{eqnarray}
	\label{e:PBWKazhdanterm}
	\Theta(y) \,=\, y + \sum_{|(\i, \j)|_e\, \le\, n_k + 2,\ |{\bf i}|+|{\bf j}|\ge 2} \lambda_{\i, \j}(y) x^\i z^\j
\end{eqnarray}
for some $\lambda_{\i, \j}(y)\in R$; see \cite[4.2]{PT21}. The monomials $x^{\i}z^{\j}$ involved in (\ref{e:PBWKazhdanterm}) will be described in more detail in Subsection~\ref{ss3}.

\section{Dealing with the orbit $\mathrm{A_5A_1}$}
\subsection{A relation in $\mathbf{\g_e(6)}$ involving four elements of weight $\mathbf{3}$}
\label{ss1}
Following \cite[p.~149]{LT2} we choose $e=e_1+e_2+e_4+e_5+e_6+e_7$. Then $$f=f_1+5f_2+8f_4+9f_5+8f_6+5f_7$$ and $h=h_1+2h_2-9h_3+2h_4+2h_5+2h_6+2h_7-9h_8.$ The Lie algebra $\g_e(0)$ consists of two commuting $\sl_2$-triples generated by $e_{\tilde{\alpha}}, f_{\tilde{\alpha}}$ and
$e':=e_{{{1\!\!\atop{}}{2\!\!\atop{}}{3\atop2}\!\!{2\!\!\atop{}}{1\!\!\atop{}}{0\!\!\atop{}}{0\!\!\atop{}}}}+e_{{{1\!\!\atop{}}{2\!\!\atop{}}{3\atop1}\!\!{2\!\!\atop{}}{1\!\!\atop{}}{1\!\!\atop{}}{0\!\!\atop{}}}}-e_{{{1\!\!\atop{}}{2\!\!\atop{}}{2\atop1}\!\!{2\!\!\atop{}}{2\!\!\atop{}}{1\!\!\atop{}}{0\!\!\atop{}}}}$, $f':=f_{{{1\!\!\atop{}}{2\!\!\atop{}}{3\atop2}\!\!{2\!\!\atop{}}{1\!\!\atop{}}{0\!\!\atop{}}{0\!\!\atop{}}}}+f_{{{1\!\!\atop{}}{2\!\!\atop{}}{3\atop1}\!\!{2\!\!\atop{}}{1\!\!\atop{}}{1\!\!\atop{}}{0\!\!\atop{}}}}-f_{{{1\!\!\atop{}}{2\!\!\atop{}}{2\atop1}\!\!{2\!\!\atop{}}{2\!\!\atop{}}{1\!\!\atop{}}{0\!\!\atop{}}}}$. The $4$-dimensional graded component $\g_e(3)$ is a direct sum of two $\g_e(0)$-modules of highest weights $(1,0)$ and $(0,1)$. As in {\it loc.\,cit.} we choose 
\begin{eqnarray}
\nonumber
v&:=&e_{{{1\!\!\atop{}}{2\!\!\atop{}}{4\atop2}\!\!{3\!\!\atop{}}{2\!\!\atop{}}{1\!\!\atop{}}{1\!\!\atop{}}}}-e_{{{1\!\!\atop{}}{2\!\!\atop{}}{3\atop2}\!\!{3\!\!\atop{}}{2\!\!\atop{}}{2\!\!\atop{}}{1\!\!\atop{}}}}+e_{{{1\!\!\atop{}}{2\!\!\atop{}}{3\atop1}\!\!{3\!\!\atop{}}{3\!\!\atop{}}{2\!\!\atop{}}{1\!\!\atop{}}}},\\
\nonumber
v'&:=&e_{{{1\!\!\atop{}}{1\!\!\atop{}}{2\atop1}\!\!{1\!\!\atop{}}{1\!\!\atop{}}{0\!\!\atop{}}{0\!\!\atop{}}}}-e_{{{0\!\!\atop{}}{1\!\!\atop{}}{2\atop1}\!\!{2\!\!\atop{}}{1\!\!\atop{}}{0\!\!\atop{}}{0\!\!\atop{}}}}-e_{{{0\!\!\atop{}}{1\!\!\atop{}}{2\atop1}\!\!{1\!\!\atop{}}{1\!\!\atop{}}{1\!\!\atop{}}{0\!\!\atop{}}}}+2e_{{{1\!\!\atop{}}{1\!\!\atop{}}{1\atop1}\!\!{1\!\!\atop{}}{1\!\!\atop{}}{1\!\!\atop{}}{0\!\!\atop{}}}}
\end{eqnarray} as corresponding highest weight vectors. Setting $v:=-[f_{\tilde{\alpha}}, u]$ and
$v':=-[f',u']$ and using the structure constants $N_{\alpha,\beta}$ tabulated in \cite[Appendix]{LS} we then check directly that
\begin{eqnarray}
	\nonumber  
u&:=&f_{{{1\!\!\atop{}}{2\!\!\atop{}}{3\atop2}\!\!{2\!\!\atop{}}{1\!\!\atop{}}{1\!\!\atop{}}{1\!\!\atop{}}}}-f_{{{1\!\!\atop{}}{2\!\!\atop{}}{3\atop1}\!\!{2\!\!\atop{}}{2\!\!\atop{}}{1\!\!\atop{}}{1\!\!\atop{}}}}+f_{{{1\!\!\atop{}}{2\!\!\atop{}}{2\atop1}\!\!{2\!\!\atop{}}{2\!\!\atop{}}{2\!\!\atop{}}{1\!\!\atop{}}}},\\
\nonumber
u'&:=&f_{{{1\!\!\atop{}}{1\!\!\atop{}}{1\atop1}\!\!{0\!\!\atop{}}{0\!\!\atop{}}{0\!\!\atop{}}{0\!\!\atop{}}}}+f_{{{1\!\!\atop{}}{1\!\!\atop{}}{1\atop0}\!\!{1\!\!\atop{}}{0\!\!\atop{}}{0\!\!\atop{}}{0\!\!\atop{}}}}+f_{{{0\!\!\atop{}}{1\!\!\atop{}}{1\atop1}\!\!{1\!\!\atop{}}{0\!\!\atop{}}{0\!\!\atop{}}{0\!\!\atop{}}}}+2f_{{{0\!\!\atop{}}{1\!\!\atop{}}{1\atop0}\!\!{1\!\!\atop{}}{1\!\!\atop{}}{0\!\!\atop{}}{0\!\!\atop{}}}}.
\end{eqnarray}
One has to keep in mind here that $$N_{\scriptstyle{{{{1\!\!\atop{}}{2\!\!\atop{}}{2\atop1}\!\!{2\!\!\atop{}}{2\!\!\atop{}}{2\!\!\atop{}}{1\!\!\atop{}}}}, {{{1\!\!\atop{}}{2\!\!\atop{}}{4\atop2}\!\!{3\!\!\atop{}}{2\!\!\atop{}}{1\!\!\atop{}}{1\!\!\atop{}}}}}}\,=\,
N_{\scriptstyle{{{{1\!\!\atop{}}{2\!\!\atop{}}{3\atop1}\!\!{2\!\!\atop{}}{2\!\!\atop{}}{2\!\!\atop{}}{1\!\!\atop{}}}}, {{{1\!\!\atop{}}{2\!\!\atop{}}{3\atop2}\!\!{3\!\!\atop{}}{2\!\!\atop{}}{1\!\!\atop{}}{1\!\!\atop{}}}}}}\,=\,N_{\scriptstyle{{{{1\!\!\atop{}}{2\!\!\atop{}}{3\atop2}\!\!{2\!\!\atop{}}{1\!\!\atop{}}{1\!\!\atop{}}{1\!\!\atop{}}}}, {{{1\!\!\atop{}}{2\!\!\atop{}}{3\atop1}\!\!{3\!\!\atop{}}{3\!\!\atop{}}{2\!\!\atop{}}{1\!\!\atop{}}}}}}\,=\,1,$$
$$N_{{\scriptstyle{{{{0\!\!\atop{}}{1\!\!\atop{}}{1\atop0}\!\!{1\!\!\atop{}}{1\!\!\atop{}}{0\!\!\atop{}}{0\!\!\atop{}}}}, {{{1\!\!\atop{}}{1\!\!\atop{}}{1\atop1}\!\!{1\!\!\atop{}}{1\!\!\atop{}}{1\!\!\atop{}}{0\!\!\atop{}}}}}}}\,=\,N_{{\scriptstyle{{{{1\!\!\atop{}}{1\!\!\atop{}}{1\atop0}\!\!{1\!\!\atop{}}{0\!\!\atop{}}{0\!\!\atop{}}{0\!\!\atop{}}}}, {{{0\!\!\atop{}}{1\!\!\atop{}}{2\atop1}\!\!{1\!\!\atop{}}{1\!\!\atop{}}{1\!\!\atop{}}{0\!\!\atop{}}}}}}}\,=\,
N_{{\scriptstyle{{{{1\!\!\atop{}}{1\!\!\atop{}}{1\atop1}\!\!{0\!\!\atop{}}{0\!\!\atop{}}{0\!\!\atop{}}{0\!\!\atop{}}}}, {{{0\!\!\atop{}}{1\!\!\atop{}}{2\atop1}\!\!{2\!\!\atop{}}{1\!\!\atop{}}{0\!\!\atop{}}{0\!\!\atop{}}}}}}}\,=\,-N_{{\scriptstyle{{{{0\!\!\atop{}}{1\!\!\atop{}}{1\atop1}\!\!{1\!\!\atop{}}{0\!\!\atop{}}{0\!\!\atop{}}{0\!\!\atop{}}}}, {{{1\!\!\atop{}}{1\!\!\atop{}}{2\atop1}\!\!{1\!\!\atop{}}{1\!\!\atop{}}{0\!\!\atop{}}{0\!\!\atop{}}}}}}}\,=\,1,$$
 and $N_{\alpha,\beta}=-N_{-\alpha,-\beta}$ for all $\alpha,\beta
\in\Phi_+$; see \cite[p.~409]{GS} and \cite[Appendix]{LS}. 
Let $$w:=e_{{{0\!\!\atop{}}{0\!\!\atop{}}{1\atop1}\!\!{1\!\!\atop{}}{0\!\!\atop{}}{0\!\!\atop{}}{0\!\!\atop{}}}}+e_{{{0\!\!\atop{}}{0\!\!\atop{}}{1\atop0}\!\!{1\!\!\atop{}}{1\!\!\atop{}}{0\!\!\atop{}}{0\!\!\atop{}}}}+e_{{{0\!\!\atop{}}{0\!\!\atop{}}{0\atop0}\!\!{1\!\!\atop{}}{1\!\!\atop{}}{1\!\!\atop{}}{0\!\!\atop{}}}}.$$ 
Since both $[u,v]$ and $[u',v']$ lie in $\g_e(6)$ and have weight $(0,0)$ with respect to $\g_e(0)$ it follows from \cite[p.~149]{LT2} that $[u,v]=a w$ and $[u',v']=b w$ for some $a,b\in\C$. Applying $\ad e_4$ to both sides of the equation $[u,v]=a w$ gives $[[e_4,u],v]+[u,[e_4,v]]=a[e_4,w]$ implying that  
$$-[[e_4,f_{{{1\!\!\atop{}}{2\!\!\atop{}}{3\atop1}\!\!{2\!\!\atop{}}{2\!\!\atop{}}{1\!\!\atop{}}{1\!\!\atop{}}}}],v]-
[u,[e_4,e_{{{1\!\!\atop{}}{2\!\!\atop{}}{3\atop2}\!\!{3\!\!\atop{}}{2\!\!\atop{}}{2\!\!\atop{}}{1\!\!\atop{}}}}]]=a [e_4,e_{{{0\!\!\atop{}}{0\!\!\atop{}}{0\atop0}\!\!{1\!\!\atop{}}{1\!\!\atop{}}{1\!\!\atop{}}{0\!\!\atop{}}}}].$$  It follows from \cite[Appendix]{LS} that $[e_4,e_{{{0\!\!\atop{}}{0\!\!\atop{}}{0\atop0}\!\!{1\!\!\atop{}}{1\!\!\atop{}}{1\!\!\atop{}}{0\!\!\atop{}}}}]=
e_{{0\!\!\atop{}}{0\!\!\atop{}}{1\atop0}\!\!{1\!\!\atop{}}{1\!\!\atop{}}{1\!\!\atop{}}{0\!\!\atop{}}}$ and
$[e_4,e_{{{1\!\!\atop{}}{2\!\!\atop{}}{3\atop2}\!\!{3\!\!\atop{}}{2\!\!\atop{}}{2\!\!\atop{}}{1\!\!\atop{}}}}]=
e_{{1\!\!\atop{}}{2\!\!\atop{}}{4\atop2}\!\!{3\!\!\atop{}}{2\!\!\atop{}}{2\!\!\atop{}}{1\!\!\atop{}}}$. Also,
$[e_{{{1\!\!\atop{}}{2\!\!\atop{}}{3\atop1}\!\!{2\!\!\atop{}}{2\!\!\atop{}}{1\!\!\atop{}}{1\!\!\atop{}}}},
	f_4]=
\varepsilon e_{{{1\!\!\atop{}}{2\!\!\atop{}}{2\atop1}\!\!{2\!\!\atop{}}{2\!\!\atop{}}{1\!\!\atop{}}{1\!\!\atop{}}}}$ for some $\varepsilon\in\{\pm 1\}$.
As $N_{\alpha_4,\,-{{1\!\!\atop{}}{2\!\!\atop{}}{3\atop1}\!\!{2\!\!\atop{}}{2\!\!\atop{}}{1\!\!\atop{}}{1\!\!\atop{}}}}=
N_{{1\!\!\atop{}}{2\!\!\atop{}}{3\atop1}\!\!{2\!\!\atop{}}{2\!\!\atop{}}{1\!\!\atop{}}{1\!\!\atop{}},\,-\alpha_4}$ by \cite[p.~409]{GS}, applying $\ad\,e_4$ to both sides of the last equation gives $[e_{{{1\!\!\atop{}}{2\!\!\atop{}}{3\atop1}\!\!{2\!\!\atop{}}{2\!\!\atop{}}{1\!\!\atop{}}{1\!\!\atop{}}}},
h_4]= \varepsilon [e_4,e_{{{1\!\!\atop{}}{2\!\!\atop{}}{2\atop1}\!\!{2\!\!\atop{}}{2\!\!\atop{}}{1\!\!\atop{}}{1\!\!\atop{}}}}]$.
In view of \cite[Appendix]{LS} this yields $-e_{{{1\!\!\atop{}}{2\!\!\atop{}}{3\atop1}\!\!{2\!\!\atop{}}{2\!\!\atop{}}{1\!\!\atop{}}{1\!\!\atop{}}}}=\varepsilon e_{{{1\!\!\atop{}}{2\!\!\atop{}}{3\atop1}\!\!{2\!\!\atop{}}{2\!\!\atop{}}{1\!\!\atop{}}{1\!\!\atop{}}}}$ forcing $\varepsilon=-1$. As a result, $$[f_{{{1\!\!\atop{}}{2\!\!\atop{}}{2\atop1}\!\!{2\!\!\atop{}}{2\!\!\atop{}}{1\!\!\atop{}}{1\!\!\atop{}}}},e_{{{1\!\!\atop{}}{2\!\!\atop{}}{3\atop1}\!\!{3\!\!\atop{}}{3\!\!\atop{}}{2\!\!\atop{}}{1\!\!\atop{}}}}]-
[f_{{{1\!\!\atop{}}{2\!\!\atop{}}{3\atop2}\!\!{2\!\!\atop{}}{1\!\!\atop{}}{1\!\!\atop{}}{1\!\!\atop{}}}},e_{{{1\!\!\atop{}}{2\!\!\atop{}}{4\atop2}\!\!{3\!\!\atop{}}{2\!\!\atop{}}{2\!\!\atop{}}{1\!\!\atop{}}}}]= a e_{{{0\!\!\atop{}}{0\!\!\atop{}}{1\atop0}\!\!{1\!\!\atop{}}{1\!\!\atop{}}{1\!\!\atop{}}{0\!\!\atop{}}}}. $$ Using \cite[Appendix]{LS} we note that	$[e_{{{0\!\!\atop{}}{0\!\!\atop{}}{1\atop0}\!\!{1\!\!\atop{}}{1\!\!\atop{}}{1\!\!\atop{}}{0\!\!\atop{}}}},e_{{{1\!\!\atop{}}{2\!\!\atop{}}{2\atop1}\!\!{2\!\!\atop{}}{2\!\!\atop{}}{1\!\!\atop{}}{1\!\!\atop{}}}}]=e_{{{1\!\!\atop{}}{2\!\!\atop{}}{3\atop1}\!\!{3\!\!\atop{}}{3\!\!\atop{}}{2\!\!\atop{}}{1\!\!\atop{}}}}$ and $[e_{{{0\!\!\atop{}}{0\!\!\atop{}}{1\atop0}\!\!{1\!\!\atop{}}{1\!\!\atop{}}{1\!\!\atop{}}{0\!\!\atop{}}}},e_{{{1\!\!\atop{}}{2\!\!\atop{}}{3\atop2}\!\!{2\!\!\atop{}}{1\!\!\atop{}}{1\!\!\atop{}}{1\!\!\atop{}}}}]=-e_{{{1\!\!\atop{}}{2\!\!\atop{}}{4\atop2}\!\!{3\!\!\atop{}}{2\!\!\atop{}}{2\!\!\atop{}}{1\!\!\atop{}}}}.$ Therefore,
$$[f_{{{1\!\!\atop{}}{2\!\!\atop{}}{2\atop1}\!\!{2\!\!\atop{}}{2\!\!\atop{}}{1\!\!\atop{}}{1\!\!\atop{}}}},
[e_{{{0\!\!\atop{}}{0\!\!\atop{}}{1\atop0}\!\!{1\!\!\atop{}}{1\!\!\atop{}}{1\!\!\atop{}}{0\!\!\atop{}}}},e_{{{1\!\!\atop{}}{2\!\!\atop{}}{2\atop1}\!\!{2\!\!\atop{}}{2\!\!\atop{}}{1\!\!\atop{}}{1\!\!\atop{}}}}]]+
[f_{{{1\!\!\atop{}}{2\!\!\atop{}}{3\atop2}\!\!{2\!\!\atop{}}{1\!\!\atop{}}{1\!\!\atop{}}{1\!\!\atop{}}}},[e_{{{0\!\!\atop{}}{0\!\!\atop{}}{1\atop0}\!\!{1\!\!\atop{}}{1\!\!\atop{}}{1\!\!\atop{}}{0\!\!\atop{}}}},e_{{{1\!\!\atop{}}{2\!\!\atop{}}{3\atop2}\!\!{2\!\!\atop{}}{1\!\!\atop{}}{1\!\!\atop{}}{1\!\!\atop{}}}}]]= a e_{{{0\!\!\atop{}}{0\!\!\atop{}}{1\atop0}\!\!{1\!\!\atop{}}{1\!\!\atop{}}{1\!\!\atop{}}{0\!\!\atop{}}}}.$$ Equivalently, 
$-[e_{{{0\!\!\atop{}}{0\!\!\atop{}}{1\atop0}\!\!{1\!\!\atop{}}{1\!\!\atop{}}{1\!\!\atop{}}{0\!\!\atop{}}}},h_{{{1\!\!\atop{}}{2\!\!\atop{}}{2\atop1}\!\!{2\!\!\atop{}}{2\!\!\atop{}}{1\!\!\atop{}}{1\!\!\atop{}}}}]-[e_{{{0\!\!\atop{}}{0\!\!\atop{}}{1\atop0}\!\!{1\!\!\atop{}}{1\!\!\atop{}}{1\!\!\atop{}}{0\!\!\atop{}}}},h_{{{1\!\!\atop{}}{2\!\!\atop{}}{3\atop2}\!\!{2\!\!\atop{}}{1\!\!\atop{}}{1\!\!\atop{}}{1\!\!\atop{}}}}]= a e_{{{0\!\!\atop{}}{0\!\!\atop{}}{1\atop0}\!\!{1\!\!\atop{}}{1\!\!\atop{}}{1\!\!\atop{}}{0\!\!\atop{}}}}$. Thus  $a=-2$ so that $$[u,v]=-2w.$$ Since $[e_4,u']=0$, applying $\ad e_4$ to both sides of the equation $[u',v']=bw$ we get $$[u', [e_4,2e_{{{1\!\!\atop{}}{1\!\!\atop{}}{1\atop1}\!\!{1\!\!\atop{}}{1\!\!\atop{}}{1\!\!\atop{}}{0\!\!\atop{}}}}]]=2[u',e_{{{1\!\!\atop{}}{1\!\!\atop{}}{2\atop1}\!\!{1\!\!\atop{}}{1\!\!\atop{}}{1\!\!\atop{}}{0\!\!\atop{}}}}]=
b[e_4,e_{{{0\!\!\atop{}}{0\!\!\atop{}}{0\atop0}\!\!{1\!\!\atop{}}{1\!\!\atop{}}{1\!\!\atop{}}{0\!\!\atop{}}}}]=b
e_{{0\!\!\atop{}}{0\!\!\atop{}}{1\atop0}\!\!{1\!\!\atop{}}{1\!\!\atop{}}{1\!\!\atop{}}{0\!\!\atop{}}}$$ (we use the fact that $N_{\alpha_4,\,{{1\!\!\atop{}}{1\!\!\atop{}}{1\atop1}\!\!{1\!\!\atop{}}{1\!\!\atop{}}{1\!\!\atop{}}{0\!\!\atop{}}}}=1$ which follows from the conventions in \cite{GS}). Our formula for $u'$ implies that $[u',e_{{{1\!\!\atop{}}{1\!\!\atop{}}{2\atop1}\!\!{1\!\!\atop{}}{1\!\!\atop{}}{1\!\!\atop{}}{0\!\!\atop{}}}}]=[f_{{{1\!\!\atop{}}{1\!\!\atop{}}{1\atop1}\!\!{0\!\!\atop{}}{0\!\!\atop{}}{0\!\!\atop{}}{0\!\!\atop{}}}},e_{{{1\!\!\atop{}}{1\!\!\atop{}}{2\atop1}\!\!{1\!\!\atop{}}{1\!\!\atop{}}{1\!\!\atop{}}{0\!\!\atop{}}}}]$. As $[e_{{{0\!\!\atop{}}{0\!\!\atop{}}{1\atop0}\!\!{1\!\!\atop{}}{1\!\!\atop{}}{1\!\!\atop{}}{0\!\!\atop{}}}},e_{{{1\!\!\atop{}}{1\!\!\atop{}}{1\atop1}\!\!{0\!\!\atop{}}{0\!\!\atop{}}{0\!\!\atop{}}{0\!\!\atop{}}}}]=-e_{{{1\!\!\atop{}}{1\!\!\atop{}}{2\atop1}\!\!{1\!\!\atop{}}{1\!\!\atop{}}{1\!\!\atop{}}{0\!\!\atop{}}}}$ by \cite[Appendix]{LS}, we now obtain $$-2[f_{{{1\!\!\atop{}}{1\!\!\atop{}}{1\atop1}\!\!{0\!\!\atop{}}{0\!\!\atop{}}{0\!\!\atop{}}{0\!\!\atop{}}}},[e_{{{0\!\!\atop{}}{0\!\!\atop{}}{1\atop0}\!\!{1\!\!\atop{}}{1\!\!\atop{}}{1\!\!\atop{}}{0\!\!\atop{}}}},e_{{{1\!\!\atop{}}{1\!\!\atop{}}{1\atop1}\!\!{0\!\!\atop{}}{0\!\!\atop{}}{0\!\!\atop{}}{0\!\!\atop{}}}}]]=2[e_{{{0\!\!\atop{}}{0\!\!\atop{}}{1\atop0}\!\!{1\!\!\atop{}}{1\!\!\atop{}}{1\!\!\atop{}}{0\!\!\atop{}}}},h_{{{1\!\!\atop{}}{1\!\!\atop{}}{1\atop1}\!\!{0\!\!\atop{}}{0\!\!\atop{}}{0\!\!\atop{}}{0\!\!\atop{}}}}]=
be_{{0\!\!\atop{}}{0\!\!\atop{}}{1\atop0}\!\!{1\!\!\atop{}}{1\!\!\atop{}}{1\!\!\atop{}}{0\!\!\atop{}}}.$$ Hence $b=2$ so that
$[u',v']=2w$. In view of the above the following relation holds in $\g_e(6)$:
\begin{equation}
	\label{relation}
[u,v]+[u',v']=0.
\end{equation}
\subsection{Searching for a quadratic relation in $\mathbf{U(\g,e)}^{\mathbf{\mathrm ab}}$}
\label{ss2}
Our hope is that despite (\ref{relation}) the element $[\Theta_u,\Theta_v]+[\Theta_{u'},\Theta_{v'}]\in U(\g,e)$ is nonzero and, moreover, lies in 
$\F_8(Q)\setminus \F_7(Q)$. 
Let $\P(\g,e)\,=\big(\gr_{\F}(U(\g,e)), \{\,\cdot\,,\,\cdot\,\}\big)$ denote the Poisson algebra associate with Kazhdan-filtered algebra $U(\g,e)$. It is well-known that $\P(\g,e)$ identifies with the algebra of regular functions on the Slodowy slice $e+\g_f$ to adjoint $G$-orbit $e$; see \cite{Pr02, GG02}. We identify $\P(\g,e)$ with the symmetric algebra $S(\g_e)$ by using the isomorphism between $\g$ and $\g^*$ induced by the $G$-invariant symmetric bilinear form
$(\,\cdot\,,\cdot\,)$ on $\g$. We write $\mathcal{I}$ for the ideal of $\P(\g,e)$ generated by $\bigcup_{i\ne 2}\g_e(i)$ and put $\bar{\P}:=\P(\g,e)/\mathcal{I}$. Obviously, $\bar{\P} \cong S(\g_2(2))$ as $\C$-algebras.

Given $y\in \g_e(i)$ we write $\theta_y$ for the $\F$-symbol of $\Theta_y$ in $\P_{i+2}(\g,e)$.
We put $\varphi:=\{\theta_u,\theta_v\}+\{\theta_{u'},\theta_{v'}\}$, an element of $\P_8$ (possibly zero), and denote by $\bar{\varphi}$ the image of $\varphi$ in $\bar{\P}$.
By \cite[p.~149]{LT2}, the graded component $\g_e(2)=\g_e(2)^{\g_e(0)}$ is spanned by $e$ and $e_1=e_{\alpha_1}$. In view of
(\ref{relation}) and \cite[Theorem~4.6(iv)]{Pr02}  
the linear part of $\varphi$ is zero and  there exist scalars $\lambda,\mu,\nu$ such that
$$\bar{\varphi}=\lambda e^2+\mu ee_1+\nu e_1^2.$$
In fact, the main results of \cite[Theorem~1.2]{PT21} imply that $\lambda,\mu,\nu \in R$.
Since it follows from \cite[Prop.~2.1]{Pr14} and \cite[5.2]{PT21} that the commutative quotient $U(\g,e)^{\rm ab}$ is generated by the image of $\Theta_e$, we wish to take a closer look at the image of
$\{\theta_u,\theta_v\}+\{\theta_{u'},\theta_{v'}\}$ in $\bar{\P}$.

By \cite[p.~149]{LT2}, the graded component $\g_e(1)$ is an irreducible $\g_e(0)$-module generated by
$e_{{{{2\!\!\atop{}}{3\!\!\atop{}}{4\atop2}\!\!{3\!\!\atop{}}{2\!\!\atop{}}{1\!\!\atop{}}{0\!\!\atop{}}}}}$, a highest weight vector of weight $(0,3)$ for $\g_e(0)$. Hence $[\g_e(1),\g_e(1)]\subseteq \g_e(2)=\g_e(2)^{\g_e(0)}$ has dimension $\le 1$. On the other hand, a rough calculation relying on the above expression of $f'$ shows that $(\ad f')^3(e_{{{{2\!\!\atop{}}{3\!\!\atop{}}{4\atop2}\!\!{3\!\!\atop{}}{2\!\!\atop{}}{1\!\!\atop{}}{0\!\!\atop{}}}}})\in Re_1$. Since in the present case $\g_e=\C e \oplus[\g_e,\g_e]$, we see that $$[\g_e,\g_e](2)\,=\,[\g_e(1),\g_e(1)]+[\g_e(0),\g_e(2)^{\g_e(0)}]\,=\,[\g_e(1),\g_e(1)]^{\g_e(0)}$$
has codimension $1$ in $\g_e(2)$. The preceding remark now entails that $e_1\in [\g_e(1),\g_e(1)]$. 
 
Since it is immediate from \cite[Prop.~2.1]{Pr14} and \cite[5.2]{PT21} that the largest commutative quotient of $U(\g,e)$ is generated by the image of $\Theta_e$ we would find a desired quadratic relation in $U(\g,e)^{\rm ab}$  if we managed to prove that the coefficient $\lambda$
of $\bar{\varphi}$ is nonzero. Indeed, let $I_c$ denote the $2$-sided ideal of $U(\g,e)$ generated by all commutators. If it happens that $\lambda\in R^\times$ then the element  $[\Theta_u,\Theta_v]+[\Theta_{u'},\Theta_{v'}]\in I_c\cap Q_R$ has Kazhdan degree $8$ and is congruent to $\lambda \Theta_e^2$ modulo $I_c\cap U(\g_R,e)+ \F_7(Q_R)$. As \cite[Prop.~5.4]{PT21} yields $$U(\g,e)\cap \F_7(Q_R)\subset R 1+R\Theta_e+I_c\cap U(\g_R,e)$$ the latter would imply that 
$\lambda\Theta_e^2+\eta\Theta_e+\xi 1\in I_c$ for some $\lambda\in R^\times$ and $\eta,\xi\in R$.

From the expression for $f$ in Subsection~\ref{ss1} we get $(e,f)=5+8+9+8+5+1=36$. As $(e,f_1)=(e_1,f_1)=1$ we obtain $(e_1,f-f_1)=0$ and $(e,f-f_0)=35$.  Since all elements of $\I$ vanish on $f-f_1$ this gives
\begin{equation}
	\label{varphi}
	\varphi(f-f_1)=\lambda(e,f-f_1)^2=5^27^2\lambda.\end{equation}  This formula indicates that we might expect some complications in characteristic $7$. 
\subsection{Computing $\mathbf{\lambda}$, part~1}
\label{ss3} 
In order to determine $\lambda$ we need a more explicit formula for commutators
$[\Theta_a,\Theta_b]$ with $a,b\in \g_e(3)$. For that purpose, it is more convenient to use the construction of $U(\g,e)$ introduced by Gan--Ginzburg in \cite{GG02}.  Let 
$\chi\in \g^*$ be such that $\chi(x)=(e,x)$ for all $x\in\g$ and set $\n':=\bigoplus_{i\le -2}\g(i)$ and $\n:=\bigoplus_{i\le 1}\g(i)$. Let $\mathcal{J}_\chi$ denote the left ideal of $U(\g)$ generated by all $x-\chi(x)$ with $x\in\n'$ and put
$\widehat{Q}:=U(\g)/\mathcal{J}_\chi$.
Since $\chi$ vanishes on $[\n,\n']\subseteq \bigoplus_{i\le -3}\g(i)$, the left ideal $\mathcal{J}_\chi$ is stable under the adjoint action of $\n$. Therefore, $\n$ acts on $\widehat{Q}$. Moreover, the fixed point space $\widehat{Q}^{\ad\n}$ carries a natural algebra structure given by $(x+{\mathcal J}_\chi)(y+\mathcal{J}_\chi)=xy+\mathcal{J}_\chi$ for all $x+\mathcal{J}_\chi,y+\mathcal{J}_\chi\in \widehat{Q}_\chi$. By \cite[Theorem~4.1]{GG02}, $U(\g,e)\cong \widehat{Q}^{\ad\n}$ as algebras. The Kazhdan filtration $\F$ of $\widehat{Q}$  (induced by that of $U(\g)$) is nonnegative.

Let $\langle\,\cdot\,,\,\cdot\,\rangle$ be the non-degenerate symplectic form on $\g(-1)$ given by $\langle x,y\rangle=(e,[x,y])$ for all $x,y\in \g(-1)$ and let $z_1,\cdots,z_s,z_{s+1},\cdots, z_{2s}$ be  a basis of $\g(-1)$ such that  $\langle z_{i+s}, z_j\rangle=\delta_{ij}$ and
$\langle z_i, z_j\rangle =\langle z_{i+s}, z_{j+s}\rangle=0$ for all $1\le i,j\le s$. Let $\p=\bigoplus_{i\ge 0}\g(i)$, the parabolic subalgebra associated with the cocharacter $\tau$, and let $x_1,\ldots, x_m$ be a homogeneous basis of $\p$ such that $x_1,\ldots, x_r$ is a basis of $\g_e\subset \p$ and $x_i\in \g(n_i)$ for some $n_i\in\Z_{\ge 0}$ (and all $i\le m$). Given $(\i,\j)\in\Z_{\ge 0}^m\times Z_{\ge 0}^{2s}$ we set $x^{\i}z^{\j}:=x_1^{i_1}\cdots x_m^{i_m}z_1^{j_1}\cdots
z_{2s}^{j_{2s}}$. Clearly, $\F_d(\widehat{Q})\subset U(\g)/\mathcal{J}_\chi$ has $\C$-basis
consisting of all $x^{\i}z^{\j}$ 
with $$|(\i,\j)|_e:=\textstyle{\sum}_{k=1}^m\, i_k(n_k+2)+\textstyle{\sum}_{k=1}^{2s}\, j_k={\rm wt}_h(x^{\i}z^{\j})+2\deg(x^{\i}z^{\j})\le d.$$

As explained in \cite[2.1]{Pr07a} the algebra $U(\g,e)$ has a PBW basis consisting of monomials
$\Theta^{\i}:=\Theta_1^{i_1}\cdots \Theta_r^{x_r}$ with $\i\in \Z_{\ge 0}^r$, where  
$$\Theta_k \,=\, x_k + \sum_{|(\i, \j)|_e\, \le\, n_k + 2,\ |{\bf i}|+|{\bf j}|\ge 2} \lambda_{\i, \j}^k x^\i z^\j,\qquad 1\le k\le r,$$ where $\lambda_{\i,\j}^k\in \C$ and $\lambda_{\i,\j}^k=0$ whenever $\j=\mathbf{0}$ and
$i_j=0$ for $j>r$. The elements $\{\Theta_k\,|\,\,1\le k\le r\}$ are unique by \cite[Lemma~2.4]{PT21}. 

Given $a=\sum_i \xi_i x_i\in \g_e$ we put $\Theta_a:=\sum_i \xi_i\Theta_i$.
Following \cite[2.4]{Pr07a} we denote by $\mathbf{A}_e$ the associative $\C$-algebra generated by $z_1,\ldots, z_s,z_{s+1},\ldots, z_{2s}$ subject to the relations $[z_{i+s},z_j]=\delta_{ij}$ and $[z_i,z_j]=[z_{i+s},z_{j+s}]=0$ for all $1\le i,j\le s$. Clearly, $\mathbf{A}_e\cong \mathbf{A}_s(\C)$, the $s$-th Weyl algebra over $\C$.
Let $i\mapsto i^*$ denote the involution of the index set $\{1,\ldots, s,s+1,\ldots, 2s\}$ such that $i^*=i+s$ for $i\le s$ and $i^*=i-s$ for $i>s$, and put $z_i^*:=(-1)^{p(i)}z_{i^*}$ where $p(i)=0$ if $i\le s$ and $p(i)=1$ if $i>s$. Then  $[z_{i}^*,z_j]\in \delta_{ij} + \mathcal{J}_\chi$ for all $i\le 2s$. 

Let $a\in \g_e(d)$ where $d\ge 1$. As $\g(-1)\subset\n$ and $\Theta_a\in \widehat{Q}^{\n}$ it is straightforward to see that 
\begin{equation*}
\Theta_a \,\equiv\, a + \sum_{i=1}^{2s}[a,z_i^*]z_i+\sum_{|(\i, {\bf 0})|_e\, =\, d + 2,\ |{\bf i}|=2} \lambda_{\i, {\bf 0}}(a) x^\i+
\sum_{|(\i, \j)|_e\, =\, d + 2,\ |{\bf i}|+|{\bf j}|\ge 3} \lambda_{\i, \j}(a) x^\i z^\j\ \ \ \ {\rm mod}\ \F_{d+1}(\widehat{Q})\end{equation*} where $\lambda_{\i.\j}(a)\in \C.$
By \cite[Prop.~2.2]{Pr07a}, there exists an injective homomorphism of $\C$-algebras   
$\widetilde{\mu}\colon\, U(\g,e)\hookrightarrow U(\p)\otimes \mathbf{A}_e^{\rm op}$ such that
$$\tilde{\mu}(\Theta_k)=x_k\otimes 1+ \sum_{|(\i, \j)|_e\, \le\, n_k + 2,\ |{\bf i}|+|{\bf j}|\ge 2} \lambda_{\i, \j}^k x^\i\otimes z^\j,\qquad 1\le k\le r.$$  If $u_1,u_2\in U(\p)$ and 
$c_1,c_2\in\mathbf{A}_e^{\rm op}$ then $$[u_1\otimes c_1,u_2\otimes c_2]=u_1u_2\otimes c_2c_1-
u_2u_1\otimes c_1c_2=u_1u_2\otimes [c_1,c_2]+[u_1,u_2]\otimes c_1c_2.$$ 

Now let $a\in\g_e(d_1)$ and 
$b\in \g_e(d_2)$, where $d_1,d_2$ are positive integers. Combining the above expressions for 
$\Theta_a$ and $\Theta_b$ with the preceding remark and properties of $\tilde{\mu}$ one observes that 
\begin{eqnarray*}
[\Theta_a,\Theta_b]&\equiv& [a,b]+\sum_{i=1}^{2s}[[a,b]z_i^*]z_i+\sum_{i=1}^{2s}[a,z_i^*][b,z_i]\\
&+&q(a,b)+\sum_{|(\i, \j)|_e\, =\, d_1+d_2 + 2,\ |{\bf i}|+|{\bf j}|\ge 3}\lambda_{\i, \j}(a,b) x^\i z^\j\ \ \ \ {\rm mod}\ \F_{d_1+d_2+1}(\widehat{Q}),
\end{eqnarray*} where  $\lambda_{\i,\j}(a,b)\in\C$ and $q(a,b)$ is a linear combination of $[a,x_i]x_j$ with $n_i+n_j=d_2+2$ and $[b,x_i]x_j$ with $n_i+n_j=d_1+2$. In view of (\ref{relation}) this implies that
$$\{\theta_u,\theta_v\}+\{\theta_{u'},\theta_{v'}\}\,=\,\sum_{i=1}^{2s}\big([u,z_i^*][v,z_i]
+[u',z_i^*][v',z_i]\big)+q(u,v,u',v')+\mbox{terms of standard degree} \ge 3,$$
where $q(u,v,u',v')=q(u,v)+q(u',v')$. 
All terms of standard degree $\ge 3$ involved in $\{\theta_u,\theta_v\}+\{\theta_{u'},\theta_{v'}\}$ have Kazhdan degree $8$.  Therefore, they must vanish at $f-f_1\in \g(-2)$. Since each quadratic monomial involved in $q(u,v,u',v')$ has a linear factor of standard degree $\ge 3$ we also have that $q(u,v,u',v')(f-f_1)=0$. As a consequence, 
\begin{equation*}
\big(\{\theta_u,\theta_v\}+\{\theta_{u'},\theta_{v'}\}\big)(f-f_1)\,=\,
\sum_{i=1}^{2s}([u,z_i^*],f-f_1)([v,z_i], f-f_1)+\sum_{i=1}^{2s}([u',z_i^*],f-f_1)([v',z_i], f-f_1).
\end{equation*}
\subsection{Computing $\mathbf{\lambda}$, part~2}
\label{ss4} Our deliberation in Subsection~\ref{ss3} show that in order to determine $\lambda$ we need to evaluate two sums:
$$A\,:=\,\sum_{i=1}^{2s}([u,z_i^*],f-f_1)([v,z_i], f-f_1)\ \,\mbox{ and }\ \, 
B\,:=\,\sum_{i=1}^{2s}([u',z_i^*],f-f_1)([v',z_i], f-f_1).$$ To simplify notation we put $E:=\ad e$,
$H:=\ad H$, $F=\ad f$ and $H_1:=\ad h_1=\ad h_{\alpha_1}$. Since $u,u',v,v'\in\g_e(3)$ there exist
$u_-\in \C F^3(u)$, $v_-\in \C F^3(v)$, $u_-'\in\C F^3(u')$ and $v_-'\in \C F^3(v')$ such that $u=E^3(u_-)$, $v=E^3(v_-)$, $u'=E^3(u'_-)$ and $v'=E^3(v_-')$. As $\g_e\subset \p$ the $\sl_2$-theory shows that the elements $u_-,v_-,u_-',v_-'$ lie in $\g_f(-3)$.  Using the $\g$-invariance of $(\,\cdot\,,\,\cdot\,)$ and the fact that $E^3(f-f_1)=0$ we get
\begin{eqnarray*}
&A&=\,\sum_{i=1}^{2s}([E^3(u_-),z_i^*],f-f_1)([E^3(v_-),z_i],f-f_1)\\
&=&\sum_{i=1}^{2s}(z_i^*,[E^3(u_-),f-f_1])([z_i,E^3(v_-)f-f_1])\\
&=&\!\!\sum_{i=1}^{2s}(z_i^*,E^3([u_-,f-f_1])-3E([E(u_-),h-h_1]))(z_i,E^3([v_-,f-f_1])-
3E([E(v_-),h-h_1]))\\
&=&\sum_{i=1}^{2s}(e,[E^2([u_-,f-f_1])-3[E(u_-),h-h_1],z_i^*])(e,[E^2([v_-,f-f_1])-
3[E(v_-),h-h_1],z_i]).	
\end{eqnarray*}
Our choice of the $z_i^*$'s implies that $\langle x,y\rangle=\sum_{i=1}^{2s}\langle z_i^*,x\rangle\langle z_i,y\rangle$ for all $x,y\in \g(-1)$.
The definition of $\langle\,\cdot\,,\cdot\,\rangle$ then yields 
\begin{eqnarray*}
A&=&\big(e,\big[[E^2([u_-,f-f_1])-3[E(u_-),h-h_1],[E^2([v_-,f-f_1])-3[E(v_-),h-h_1]\big]\big)\\
&=&\big([E^3(u_-),f-f_1],E^2([v_-,f-f_1]))-3[E(v_-),h-h_1]\big)\\
&=&\big([u,f-f_1],E^2([v_-,f-f_1]\big)-3\big([u,f-f_1],[E(v_-),h-h_1]\big)\\
&=&2([u,e_1],[v_-,f-f_1])-3(u,[f-f_1,[[e,v_-],h-h_1])\,=\,2([[u,e_1],f_1],v_-)-\\
&-&3(u,[[-h+h_1,v_-],h-h_1])+3(u,[[e,[f_1,v_-]],h-h_1])-6(u,[[e,v_-],f-f_1])\\
&=&2([[u,e_1],f_1],v_-)-3\big(u,(H-H_1)^2(v_-)\big)+3([e,[u,[f_1,v_-]],h-h_1)
-6([e,[u,v_-]],f-f_1)\\
&=&2([[u,e_1],f_1],v_-)-3\big(u,(H-H_1)^2(v_-)\big)-6([f_1,v_-],[u,e-e_1])+6([u,v_-],h-h_1)\\
&=&2([[u,e_1],f_1],v_-)-3\big(u,(H-H_1)^2(v_-)\big)+6(u,[e_1,[f_1,v_-]])-6(u,[h-h_1,v_-])\\
&=&8([[u,e_1],f_1],v_-)-3\big((H-H_1)(H-H_1-2)(u),v_-\big).
\end{eqnarray*}	
Absolutely similarly we obtain that
$$B\,=\,8([[u',e_1],f_1],v_-')-3\big((H-H_1)(H-H_1-2)(u'),v_-'\big).$$
The expression for $v$ in Subsection~\ref{ss1} yields $[e_1,u]=[h_1,u]=0$. Since $[h,u]=3u$ this implies that $A=-9(u,v_-)$. Also, $u'=u_1'+u_2'$ where
$$u_1'=f_{{{1\!\!\atop{}}{1\!\!\atop{}}{1\atop1}\!\!{0\!\!\atop{}}{0\!\!\atop{}}{0\!\!\atop{}}{0\!\!\atop{}}}}+f_{{{1\!\!\atop{}}{1\!\!\atop{}}{1\atop0}\!\!{1\!\!\atop{}}{0\!\!\atop{}}{0\!\!\atop{}}{0\!\!\atop{}}}}\ \ \mbox{and}\ \ 
u_2'=f_{{{0\!\!\atop{}}{1\!\!\atop{}}{1\atop1}\!\!{1\!\!\atop{}}{0\!\!\atop{}}{0\!\!\atop{}}{0\!\!\atop{}}}}+2f_{{{0\!\!\atop{}}{1\!\!\atop{}}{1\atop0}\!\!{1\!\!\atop{}}{1\!\!\atop{}}{0\!\!\atop{}}{0\!\!\atop{}}}}.$$ As $[e_1,u_2']=[f_1,u_1']=0$ and $[h_1,u_1']=-u_1'$ we have $[[u',e_1],f_1]=[[u_1',e_1],f_1]=[u_1',h_1]=u_1'$. As $[h_1,u_2']=u_2'$ and $u_1',u_2'\in \g(3)$ we have $$(H-H_1)(H-H_1-2)(u')=(H_1-3)(H_1-1)(u')=-2(H_1-3)(u_1')=8u_1'.$$
From this it is immediate that $B=8(u_1',v_-')-24(u_1',v_-')=-16(u_1',v_-')$ and 
$$A+B=-9(u,v_-)-16(u_1',v_-').$$
Recall that $v=E^3(v_-)$ and $v'=E^3(v_-')$. Since both $v$ and $v'$ have weight $3$ it is straightforward to check that $v_-=\frac{1}{36}F^3(v)$ and $v_-'-=\frac{1}{36}F^3(v')$. As a result,
$$36(A+B)\,=\,9\big((F^3(u),v\big)+16\big((F^3(u_1'),v'\big).$$

Our next step is to compute $F^3(u_1)=(\ad f)^3(f_{{{1\!\!\atop{}}{1\!\!\atop{}}{1\atop1}\!\!{0\!\!\atop{}}{0\!\!\atop{}}{0\!\!\atop{}}{0\!\!\atop{}}}}+f_{{{1\!\!\atop{}}{1\!\!\atop{}}{1\atop0}\!\!{1\!\!\atop{}}{0\!\!\atop{}}{0\!\!\atop{}}{0\!\!\atop{}}}})$. The formula for $f$ in Subsection~\ref{ss1} shows that
$$[f,u_1']=9[f_5,f_{{{1\!\!\atop{}}{1\!\!\atop{}}{1\atop1}\!\!{0\!\!\atop{}}{0\!\!\atop{}}{0\!\!\atop{}}{0\!\!\atop{}}}}]+
5[f_2,f_{{{1\!\!\atop{}}{1\!\!\atop{}}{1\atop0}\!\!{1\!\!\atop{}}{0\!\!\atop{}}{0\!\!\atop{}}{0\!\!\atop{}}}}]+8[f_6,f_{{{1\!\!\atop{}}{1\!\!\atop{}}{1\atop0}\!\!{1\!\!\atop{}}{0\!\!\atop{}}{0\!\!\atop{}}{0\!\!\atop{}}}}].$$ Using the structure constants
and conventions of \cite{GS} we get
$[f,u'_1]=4(f_{{{1\!\!\atop{}}{1\!\!\atop{}}{1\atop1}\!\!{1\!\!\atop{}}{0\!\!\atop{}}{0\!\!\atop{}}{0\!\!\atop{}}}}+
2f_{{{1\!\!\atop{}}{1\!\!\atop{}}{1\atop0}\!\!{1\!\!\atop{}}{1\!\!\atop{}}{0\!\!\atop{}}{0\!\!\atop{}}}})$.
Then
 \begin{eqnarray*}[f,[f,u'_1]]&=&
 	4\big(10[f_1,f_{{{1\!\!\atop{}}{1\!\!\atop{}}{1\atop0}\!\!{1\!\!\atop{}}{1\!\!\atop{}}{0\!\!\atop{}}{0\!\!\atop{}}}}]
 	+8[f_4,f_{{{1\!\!\atop{}}{1\!\!\atop{}}{1\atop1}\!\!{1\!\!\atop{}}{0\!\!\atop{}}{0\!\!\atop{}}{0\!\!\atop{}}}}]
+8[f_6,f_{{{1\!\!\atop{}}{1\!\!\atop{}}{1\atop1}\!\!{1\!\!\atop{}}{0\!\!\atop{}}{0\!\!\atop{}}{0\!\!\atop{}}}}]
+8[f_7,f_{{{1\!\!\atop{}}{1\!\!\atop{}}{1\atop0}\!\!{1\!\!\atop{}}{1\!\!\atop{}}{0\!\!\atop{}}{0\!\!\atop{}}}}]\big)\\
&=&4(-10f_{{{1\!\!\atop{}}{1\!\!\atop{}}{1\atop1}\!\!{1\!\!\atop{}}{1\!\!\atop{}}{0\!\!\atop{}}{0\!\!\atop{}}}}
-8f_{{{1\!\!\atop{}}{1\!\!\atop{}}{2\atop1}\!\!{1\!\!\atop{}}{0\!\!\atop{}}{0\!\!\atop{}}{0\!\!\atop{}}}}
+8f_{{{1\!\!\atop{}}{1\!\!\atop{}}{1\atop1}\!\!{1\!\!\atop{}}{1\!\!\atop{}}{0\!\!\atop{}}{0\!\!\atop{}}}}
+10f_{{{1\!\!\atop{}}{1\!\!\atop{}}{1\atop0}\!\!{1\!\!\atop{}}{1\!\!\atop{}}{1\!\!\atop{}}{0\!\!\atop{}}}}\big)\\
&=&8(-f_{{{1\!\!\atop{}}{1\!\!\atop{}}{1\atop1}\!\!{1\!\!\atop{}}{1\!\!\atop{}}{0\!\!\atop{}}{0\!\!\atop{}}}}
-4f_{{{1\!\!\atop{}}{1\!\!\atop{}}{2\atop1}\!\!{1\!\!\atop{}}{0\!\!\atop{}}{0\!\!\atop{}}{0\!\!\atop{}}}}
+5f_{{{1\!\!\atop{}}{1\!\!\atop{}}{1\atop0}\!\!{1\!\!\atop{}}{1\!\!\atop{}}{1\!\!\atop{}}{0\!\!\atop{}}}}).
\end{eqnarray*}
Finally,
\begin{eqnarray*}
F^3(u_1')&=&
8\big(-8[f_4,f_{{{1\!\!\atop{}}{1\!\!\atop{}}{1\atop1}\!\!{1\!\!\atop{}}{1\!\!\atop{}}{0\!\!\atop{}}{0\!\!\atop{}}}}]
-5[f_7,f_{{{1\!\!\atop{}}{1\!\!\atop{}}{1\atop1}\!\!{1\!\!\atop{}}{1\!\!\atop{}}{0\!\!\atop{}}{0\!\!\atop{}}}}]
-32[f_6,f_{{{1\!\!\atop{}}{1\!\!\atop{}}{2\atop1}\!\!{1\!\!\atop{}}{0\!\!\atop{}}{0\!\!\atop{}}{0\!\!\atop{}}}}]
+25[f_2,f_{{{1\!\!\atop{}}{1\!\!\atop{}}{1\atop0}\!\!{1\!\!\atop{}}{1\!\!\atop{}}{1\!\!\atop{}}{0\!\!\atop{}}}}]
 \big)\\
 &=&8(8f_{{{1\!\!\atop{}}{1\!\!\atop{}}{2\atop1}\!\!{1\!\!\atop{}}{1\!\!\atop{}}{0\!\!\atop{}}{0\!\!\atop{}}}}
 -5f_{{{1\!\!\atop{}}{1\!\!\atop{}}{1\atop1}\!\!{1\!\!\atop{}}{1\!\!\atop{}}{1\!\!\atop{}}{0\!\!\atop{}}}}
 -32f_{{{1\!\!\atop{}}{1\!\!\atop{}}{2\atop1}\!\!{1\!\!\atop{}}{1\!\!\atop{}}{0\!\!\atop{}}{0\!\!\atop{}}}}
 -25f_{{{1\!\!\atop{}}{1\!\!\atop{}}{1\atop1}\!\!{1\!\!\atop{}}{1\!\!\atop{}}{1\!\!\atop{}}{0\!\!\atop{}}}})\\
 &=&-48(4f_{{{1\!\!\atop{}}{1\!\!\atop{}}{2\atop1}\!\!{1\!\!\atop{}}{1\!\!\atop{}}{0\!\!\atop{}}{0\!\!\atop{}}}}
 +5f_{{{1\!\!\atop{}}{1\!\!\atop{}}{1\atop1}\!\!{1\!\!\atop{}}{1\!\!\atop{}}{1\!\!\atop{}}{0\!\!\atop{}}}}).
 \end{eqnarray*}
 Therefore,$$
 \big(F^3(u_1'),v'\big)=-
 48(4f_{{{1\!\!\atop{}}{1\!\!\atop{}}{2\atop1}\!\!{1\!\!\atop{}}{1\!\!\atop{}}{0\!\!\atop{}}{0\!\!\atop{}}}}
 +5f_{{{1\!\!\atop{}}{1\!\!\atop{}}{1\atop1}\!\!{1\!\!\atop{}}{1\!\!\atop{}}{1\!\!\atop{}}{0\!\!\atop{}}}},
 e_{{{1\!\!\atop{}}{1\!\!\atop{}}{2\atop1}\!\!{1\!\!\atop{}}{1\!\!\atop{}}{0\!\!\atop{}}{0\!\!\atop{}}}}+2e_{{{1\!\!\atop{}}{1\!\!\atop{}}{1\atop1}\!\!{1\!\!\atop{}}{1\!\!\atop{}}{1\!\!\atop{}}{0\!\!\atop{}}}})=-2^5\cdot 3 \cdot 7.$$
 Next we determine $F^3(u)=(\ad f)^3(f_{{{1\!\!\atop{}}{2\!\!\atop{}}{3\atop2}\!\!{2\!\!\atop{}}{1\!\!\atop{}}{1\!\!\atop{}}{1\!\!\atop{}}}}-
 f_{{{1\!\!\atop{}}{2\!\!\atop{}}{3\atop1}\!\!{2\!\!\atop{}}{2\!\!\atop{}}{1\!\!\atop{}}{1\!\!\atop{}}}}+
 f_{{{1\!\!\atop{}}{2\!\!\atop{}}{2\atop1}\!\!{2\!\!\atop{}}{2\!\!\atop{}}{2\!\!\atop{}}{1\!\!\atop{}}}})$. Here we use conventions of \cite{GS} and the structure constants from \cite[Appendix]{LS}. We have
 \begin{eqnarray*}
[f,u]&=&8[f_6,f_{{{1\!\!\atop{}}{2\!\!\atop{}}{3\atop2}\!\!{2\!\!\atop{}}{1\!\!\atop{}}{1\!\!\atop{}}{1\!\!\atop{}}}}]
 -5[f_2,f_{{{1\!\!\atop{}}{2\!\!\atop{}}{3\atop1}\!\!{2\!\!\atop{}}{2\!\!\atop{}}{1\!\!\atop{}}{1\!\!\atop{}}}}]
-5[f_7,f_{{{1\!\!\atop{}}{2\!\!\atop{}}{3\atop1}\!\!{2\!\!\atop{}}{2\!\!\atop{}}{1\!\!\atop{}}{1\!\!\atop{}}}}]
-9[f_5,f_{{{1\!\!\atop{}}{2\!\!\atop{}}{3\atop1}\!\!{2\!\!\atop{}}{2\!\!\atop{}}{1\!\!\atop{}}{1\!\!\atop{}}}}]
+8[f_4, f_{{{1\!\!\atop{}}{2\!\!\atop{}}{2\atop1}\!\!{2\!\!\atop{}}{2\!\!\atop{}}{2\!\!\atop{}}{1\!\!\atop{}}}}]\\
&=&-8f_{{{1\!\!\atop{}}{2\!\!\atop{}}{3\atop2}\!\!{2\!\!\atop{}}{2\!\!\atop{}}{1\!\!\atop{}}{1\!\!\atop{}}}}
+5f_{{{1\!\!\atop{}}{2\!\!\atop{}}{3\atop2}\!\!{2\!\!\atop{}}{2\!\!\atop{}}{1\!\!\atop{}}{1\!\!\atop{}}}}
+5f_{{{1\!\!\atop{}}{2\!\!\atop{}}{3\atop1}\!\!{2\!\!\atop{}}{2\!\!\atop{}}{2\!\!\atop{}}{1\!\!\atop{}}}}
+9f_{{{1\!\!\atop{}}{2\!\!\atop{}}{3\atop1}\!\!{3\!\!\atop{}}{2\!\!\atop{}}{1\!\!\atop{}}{1\!\!\atop{}}}}
-8f_{{{1\!\!\atop{}}{2\!\!\atop{}}{3\atop1}\!\!{2\!\!\atop{}}{2\!\!\atop{}}{2\!\!\atop{}}{1\!\!\atop{}}}}\\
&=&-3(f_{{{1\!\!\atop{}}{2\!\!\atop{}}{3\atop2}\!\!{2\!\!\atop{}}{2\!\!\atop{}}{1\!\!\atop{}}{1\!\!\atop{}}}}
-3f_{{{1\!\!\atop{}}{2\!\!\atop{}}{3\atop1}\!\!{3\!\!\atop{}}{2\!\!\atop{}}{1\!\!\atop{}}{1\!\!\atop{}}}}
+f_{{{1\!\!\atop{}}{2\!\!\atop{}}{3\atop1}\!\!{2\!\!\atop{}}{2\!\!\atop{}}{2\!\!\atop{}}{1\!\!\atop{}}}}).
\end{eqnarray*}
Then
\begin{eqnarray*}[f,[f,u]]&=&
-3\big(9[f_5,f_{{{1\!\!\atop{}}{2\!\!\atop{}}{3\atop2}\!\!{2\!\!\atop{}}{2\!\!\atop{}}{1\!\!\atop{}}{1\!\!\atop{}}}}]
+5[f_7,f_{{{1\!\!\atop{}}{2\!\!\atop{}}{3\atop2}\!\!{2\!\!\atop{}}{2\!\!\atop{}}{1\!\!\atop{}}{1\!\!\atop{}}}}]
-15[f_2,f_{{{1\!\!\atop{}}{2\!\!\atop{}}{3\atop1}\!\!{3\!\!\atop{}}{2\!\!\atop{}}{1\!\!\atop{}}{1\!\!\atop{}}}}]
-15[f_7,f_{{{1\!\!\atop{}}{2\!\!\atop{}}{3\atop1}\!\!{3\!\!\atop{}}{2\!\!\atop{}}{1\!\!\atop{}}{1\!\!\atop{}}}}]+\\
&+&5[f_2,f_{{{1\!\!\atop{}}{2\!\!\atop{}}{3\atop1}\!\!{2\!\!\atop{}}{2\!\!\atop{}}{2\!\!\atop{}}{1\!\!\atop{}}}}]
+9[f_5,f_{{{1\!\!\atop{}}{2\!\!\atop{}}{3\atop1}\!\!{2\!\!\atop{}}{2\!\!\atop{}}{2\!\!\atop{}}{1\!\!\atop{}}}}]\big)\,=\,
-3(-9f_{{{1\!\!\atop{}}{2\!\!\atop{}}{3\atop2}\!\!{3\!\!\atop{}}{2\!\!\atop{}}{1\!\!\atop{}}{1\!\!\atop{}}}} 
-5f_{{{1\!\!\atop{}}{2\!\!\atop{}}{3\atop2}\!\!{2\!\!\atop{}}{2\!\!\atop{}}{2\!\!\atop{}}{1\!\!\atop{}}}}-\\
&-&5f_{{{1\!\!\atop{}}{2\!\!\atop{}}{3\atop2}\!\!{2\!\!\atop{}}{2\!\!\atop{}}{2\!\!\atop{}}{1\!\!\atop{}}}}
-9f_{{{1\!\!\atop{}}{2\!\!\atop{}}{3\atop1}\!\!{3\!\!\atop{}}{2\!\!\atop{}}{2\!\!\atop{}}{1\!\!\atop{}}}}
+15f_{{{1\!\!\atop{}}{2\!\!\atop{}}{3\atop2}\!\!{3\!\!\atop{}}{2\!\!\atop{}}{1\!\!\atop{}}{1\!\!\atop{}}}}
+15f_{{{1\!\!\atop{}}{2\!\!\atop{}}{3\atop1}\!\!{3\!\!\atop{}}{2\!\!\atop{}}{2\!\!\atop{}}{1\!\!\atop{}}}}
)\\
&=&6(5f_{{{1\!\!\atop{}}{2\!\!\atop{}}{3\atop2}\!\!{2\!\!\atop{}}{2\!\!\atop{}}{2\!\!\atop{}}{1\!\!\atop{}}}}
-3f_{{{1\!\!\atop{}}{2\!\!\atop{}}{3\atop1}\!\!{3\!\!\atop{}}{2\!\!\atop{}}{2\!\!\atop{}}{1\!\!\atop{}}}}
-3f_{{{1\!\!\atop{}}{2\!\!\atop{}}{3\atop2}\!\!{3\!\!\atop{}}{2\!\!\atop{}}{1\!\!\atop{}}{1\!\!\atop{}}}}).	
\end{eqnarray*}
Finally,
\begin{eqnarray*}
F^3(u)&=&6\big(45
[f_5,f_{{{1\!\!\atop{}}{2\!\!\atop{}}{3\atop2}\!\!{2\!\!\atop{}}{2\!\!\atop{}}{2\!\!\atop{}}{1\!\!\atop{}}}}]
-15[f_2,f_{{{1\!\!\atop{}}{2\!\!\atop{}}{3\atop1}\!\!{3\!\!\atop{}}{2\!\!\atop{}}{2\!\!\atop{}}{1\!\!\atop{}}}}]-
24[f_6,f_{{{1\!\!\atop{}}{2\!\!\atop{}}{3\atop1}\!\!{3\!\!\atop{}}{2\!\!\atop{}}{2\!\!\atop{}}{1\!\!\atop{}}}}]
-24[f_4,f_{{{1\!\!\atop{}}{2\!\!\atop{}}{3\atop2}\!\!{3\!\!\atop{}}{2\!\!\atop{}}{1\!\!\atop{}}{1\!\!\atop{}}}}]-\\
&-&15[f_7,f_{{{1\!\!\atop{}}{2\!\!\atop{}}{3\atop2}\!\!{3\!\!\atop{}}{2\!\!\atop{}}{1\!\!\atop{}}{1\!\!\atop{}}}}]
\big)\,=\,6(-45f_{{{1\!\!\atop{}}{2\!\!\atop{}}{3\atop2}\!\!{3\!\!\atop{}}{2\!\!\atop{}}{2\!\!\atop{}}{1\!\!\atop{}}}}
+15f_{{{1\!\!\atop{}}{2\!\!\atop{}}{3\atop2}\!\!{3\!\!\atop{}}{2\!\!\atop{}}{2\!\!\atop{}}{1\!\!\atop{}}}}
+24f_{{{1\!\!\atop{}}{2\!\!\atop{}}{3\atop1}\!\!{3\!\!\atop{}}{3\!\!\atop{}}{2\!\!\atop{}}{1\!\!\atop{}}}}
+24f_{{{1\!\!\atop{}}{2\!\!\atop{}}{4\atop2}\!\!{3\!\!\atop{}}{2\!\!\atop{}}{1\!\!\atop{}}{1\!\!\atop{}}}}+\\
&+&15f_{{{1\!\!\atop{}}{2\!\!\atop{}}{3\atop2}\!\!{3\!\!\atop{}}{2\!\!\atop{}}{2\!\!\atop{}}{1\!\!\atop{}}}})\,=\,
18(-5f_{{{1\!\!\atop{}}{2\!\!\atop{}}{3\atop2}\!\!{3\!\!\atop{}}{2\!\!\atop{}}{2\!\!\atop{}}{1\!\!\atop{}}}}
+8f_{{{1\!\!\atop{}}{2\!\!\atop{}}{3\atop1}\!\!{3\!\!\atop{}}{3\!\!\atop{}}{2\!\!\atop{}}{1\!\!\atop{}}}}
+8f_{{{1\!\!\atop{}}{2\!\!\atop{}}{4\atop2}\!\!{3\!\!\atop{}}{2\!\!\atop{}}{1\!\!\atop{}}{1\!\!\atop{}}}}).		
\end{eqnarray*}
Therefore,
\begin{eqnarray*}
\big(F^3(u),v\big)&=&18((-5f_{{{1\!\!\atop{}}{2\!\!\atop{}}{3\atop2}\!\!{3\!\!\atop{}}{2\!\!\atop{}}{2\!\!\atop{}}{1\!\!\atop{}}}}
+8f_{{{1\!\!\atop{}}{2\!\!\atop{}}{3\atop1}\!\!{3\!\!\atop{}}{3\!\!\atop{}}{2\!\!\atop{}}{1\!\!\atop{}}}}
+8f_{{{1\!\!\atop{}}{2\!\!\atop{}}{4\atop2}\!\!{3\!\!\atop{}}{2\!\!\atop{}}{1\!\!\atop{}}{1\!\!\atop{}}}},e_{{{1\!\!\atop{}}{2\!\!\atop{}}{4\atop2}\!\!{3\!\!\atop{}}{2\!\!\atop{}}{1\!\!\atop{}}{1\!\!\atop{}}}}-e_{{{1\!\!\atop{}}{2\!\!\atop{}}{3\atop2}\!\!{3\!\!\atop{}}{2\!\!\atop{}}{2\!\!\atop{}}{1\!\!\atop{}}}}+e_{{{1\!\!\atop{}}{2\!\!\atop{}}{3\atop1}\!\!{3\!\!\atop{}}{3\!\!\atop{}}{2\!\!\atop{}}{1\!\!\atop{}}}})\\
&=&18(5+8+8)=2\cdot 3^3\cdot7.
\end{eqnarray*}
As a result,
$36(A+B)=-16\cdot 2^5\cdot 3\cdot 7+9\cdot 2\cdot 3^3\cdot 7=6\cdot 7\cdot (9^2-16^2)=-6\cdot5^2\cdot 7^2.$ In view of (\ref{varphi}) we now deduce that $5^2\cdot 7^2\lambda=
A+B=-\frac{1}{6}\cdot 5^2\cdot 7^2$ forcing $\lambda=-\frac{1}{6}$. This enables us to conclude that in the present case $\dim U(\g,e)^{\rm ab}=2$. It is quite remarkable that $7^2$ gets canceled and we obtain $\lambda\in R^\times$ at the end!
 \begin{Remark}
	For safety, we have used GAP \cite{GAP} to double-check our computations and obtained the same result; i.e. $36(A+B)=-6\cdot5^2\cdot 7^2$.\footnote{The relevant code is available at \url{https://github.com/davistem/the_number_of_multiplicity-free_primitive_ideals/} }
\end{Remark}

\section{Dealing with the orbit $\mathrm{D_5(a_1)A_2}$} 
\subsection{A relation in $\mathbf{\g_e(6)}$ involving two elements of weight $\mathbf{3}$}
\label{ss5}
Following \cite[p.~150]{LT2} we choose $e=e_1+e_2+e_3+e_5+e_7+e_8+e_{\alpha_2+\alpha_4}+e_{\alpha_4+\alpha_5}$ where 
 Hence $e_{\alpha_2+\alpha_4}=[e_2,e_4]$ and $e_{\alpha_4+\alpha_5}=[e_4,e_5]$.
Then $h=6h_1+7h_2+10h_3+12h_4+7h_5+2h_7+2h_8.$
As $f_{\alpha_2+\alpha_4}=-[f_2,f_4]$ and $f_{\alpha_4+\alpha_5}=-[f_4,f_5]$ by the conventions of \cite{GS} a direct verification shows that
$$f=6f_1+f_2+10f_3+f_5+2f_7+2f_8-6[f_2,f_4]-6[f_4,f_5].$$
Therefore, $(e,f)=6+1+10+1+2+2+6+6=34.$ 
 
The Lie algebra $\g_e(0)\cong \sl(2)$ is spanned by
$e':=e_{{{1\!\!\atop{}}{2\!\!\atop{}}{3\atop1}\!\!{2\!\!\atop{}}{2\!\!\atop{}}{2\!\!\atop{}}{1\!\!\atop{}}}}
-2e_{{{1\!\!\atop{}}{2\!\!\atop{}}{3\atop2}\!\!{3\!\!\atop{}}{2\!\!\atop{}}{1\!\!\atop{}}{0\!\!\atop{}}}}
-e_{{{1\!\!\atop{}}{2\!\!\atop{}}{3\atop2}\!\!{2\!\!\atop{}}{2\!\!\atop{}}{1\!\!\atop{}}{1\!\!\atop{}}}}
-e_{{{1\!\!\atop{}}{2\!\!\atop{}}{3\atop1}\!\!{3\!\!\atop{}}{2\!\!\atop{}}{1\!\!\atop{}}{1\!\!\atop{}}}}$, $f':=2f_{{{1\!\!\atop{}}{2\!\!\atop{}}{3\atop1}\!\!{2\!\!\atop{}}{2\!\!\atop{}}{2\!\!\atop{}}{1\!\!\atop{}}}}
-f_{{{1\!\!\atop{}}{2\!\!\atop{}}{3\atop2}\!\!{3\!\!\atop{}}{2\!\!\atop{}}{1\!\!\atop{}}{0\!\!\atop{}}}}
-f_{{{1\!\!\atop{}}{2\!\!\atop{}}{3\atop2}\!\!{2\!\!\atop{}}{2\!\!\atop{}}{1\!\!\atop{}}{1\!\!\atop{}}}}
-f_{{{1\!\!\atop{}}{2\!\!\atop{}}{3\atop1}\!\!{3\!\!\atop{}}{2\!\!\atop{}}{1\!\!\atop{}}{1\!\!\atop{}}}}$ and $h':=2\varpi_6^\vee$ where $\varpi^\vee_6(e_i)=\delta_{i,6}\, e_i$ for $1\le i\le 8$. The $4$-dimensional graded component $\g_e(3)$ 
is a direct sum of two $\g_e(0)$-modules of highest weights $1$. As in {\it loc.\,cit.} we choose 
\begin{eqnarray*}
	u&:=&e_{{{1\!\!\atop{}}{2\!\!\atop{}}{2\atop1}\!\!{1\!\!\atop{}}{1\!\!\atop{}}{1\!\!\atop{}}{0\!\!\atop{}}}}
	+e_{{{1\!\!\atop{}}{1\!\!\atop{}}{2\atop1}\!\!{1\!\!\atop{}}{1\!\!\atop{}}{1\!\!\atop{}}{1\!\!\atop{}}}}
	-e_{{{1\!\!\atop{}}{2\!\!\atop{}}{2\atop1}\!\!{2\!\!\atop{}}{1\!\!\atop{}}{0\!\!\atop{}}{0\!\!\atop{}}}}
	-2e_{{{1\!\!\atop{}}{1\!\!\atop{}}{2\atop1}\!\!{2\!\!\atop{}}{1\!\!\atop{}}{1\!\!\atop{}}{0\!\!\atop{}}}}
	+3e_{{{1\!\!\atop{}}{1\!\!\atop{}}{1\atop1}\!\!{1\!\!\atop{}}{1\!\!\atop{}}{1\!\!\atop{}}{1\!\!\atop{}}}}
	+e_{{{1\!\!\atop{}}{2\!\!\atop{}}{3\atop1}\!\!{2\!\!\atop{}}{1\!\!\atop{}}{0\!\!\atop{}}{0\!\!\atop{}}}}
	\end{eqnarray*} as a highest weight vector of one of these modules and set $v:=[f',u]$. 
By standard properties of the root system $\Phi$,
\begin{eqnarray*}
v&=&2[f_{{{1\!\!\atop{}}{2\!\!\atop{}}{3\atop1}\!\!{2\!\!\atop{}}{2\!\!\atop{}}{2\!\!\atop{}}{1\!\!\atop{}}}},
e_{{{1\!\!\atop{}}{2\!\!\atop{}}{2\atop1}\!\!{1\!\!\atop{}}{1\!\!\atop{}}{1\!\!\atop{}}{0\!\!\atop{}}}}]+
2[f_{{{1\!\!\atop{}}{2\!\!\atop{}}{3\atop1}\!\!{2\!\!\atop{}}{2\!\!\atop{}}{2\!\!\atop{}}{1\!\!\atop{}}}},
e_{{{1\!\!\atop{}}{1\!\!\atop{}}{2\atop1}\!\!{1\!\!\atop{}}{1\!\!\atop{}}{1\!\!\atop{}}{1\!\!\atop{}}}}]
+[f_{{{1\!\!\atop{}}{2\!\!\atop{}}{3\atop2}\!\!{3\!\!\atop{}}{2\!\!\atop{}}{1\!\!\atop{}}{0\!\!\atop{}}}},	e_{{{1\!\!\atop{}}{2\!\!\atop{}}{2\atop1}\!\!{2\!\!\atop{}}{1\!\!\atop{}}{0\!\!\atop{}}{0\!\!\atop{}}}}]+\\
&+&2[f_{{{1\!\!\atop{}}{2\!\!\atop{}}{3\atop2}\!\!{3\!\!\atop{}}{2\!\!\atop{}}{1\!\!\atop{}}{0\!\!\atop{}}}},	e_{{{1\!\!\atop{}}{1\!\!\atop{}}{2\atop1}\!\!{2\!\!\atop{}}{1\!\!\atop{}}{1\!\!\atop{}}{0\!\!\atop{}}}}]
-[f_{{{1\!\!\atop{}}{2\!\!\atop{}}{3\atop2}\!\!{2\!\!\atop{}}{2\!\!\atop{}}{1\!\!\atop{}}{1\!\!\atop{}}}},
e_{{{1\!\!\atop{}}{1\!\!\atop{}}{2\atop1}\!\!{1\!\!\atop{}}{1\!\!\atop{}}{1\!\!\atop{}}{1\!\!\atop{}}}}]
-3[f_{{{1\!\!\atop{}}{2\!\!\atop{}}{3\atop2}\!\!{2\!\!\atop{}}{2\!\!\atop{}}{1\!\!\atop{}}{1\!\!\atop{}}}},
e_{{{1\!\!\atop{}}{1\!\!\atop{}}{1\atop1}\!\!{1\!\!\atop{}}{1\!\!\atop{}}{1\!\!\atop{}}}}]+\\
&+&[f_{{{1\!\!\atop{}}{2\!\!\atop{}}{3\atop1}\!\!{3\!\!\atop{}}{2\!\!\atop{}}{1\!\!\atop{}}{1\!\!\atop{}}}},
e_{{{1\!\!\atop{}}{2\!\!\atop{}}{2\atop1}\!\!{2\!\!\atop{}}{1\!\!\atop{}}{0\!\!\atop{}}{0\!\!\atop{}}}}]
-[f_{{{1\!\!\atop{}}{2\!\!\atop{}}{3\atop1}\!\!{3\!\!\atop{}}{2\!\!\atop{}}{1\!\!\atop{}}{1\!\!\atop{}}}},
e_{{{1\!\!\atop{}}{2\!\!\atop{}}{3\atop1}\!\!{2\!\!\atop{}}{1\!\!\atop{}}{0\!\!\atop{}}{0\!\!\atop{}}}}].
\end{eqnarray*}
From \cite[Appendix]{LS} we get
 $$N_{\scriptstyle{{{{0\!\!\atop{}}{0\!\!\atop{}}{1\atop0}\!\!{1\!\!\atop{}}{1\!\!\atop{}}{1\!\!\atop{}}{1\!\!\atop{}}}}, {{{1\!\!\atop{}}{2\!\!\atop{}}{2\atop1}\!\!{1\!\!\atop{}}{1\!\!\atop{}}{1\!\!\atop{}}{0\!\!\atop{}}}}}}\,=\,
N_{\scriptstyle{{{{0\!\!\atop{}}{1\!\!\atop{}}{1\atop0}\!\!{1\!\!\atop{}}{1\!\!\atop{}}{1\!\!\atop{}}{0\!\!\atop{}}}}, {{{1\!\!\atop{}}{1\!\!\atop{}}{2\atop1}\!\!{1\!\!\atop{}}{1\!\!\atop{}}{1\!\!\atop{}}{1\!\!\atop{}}}}}}\,=\,
N_{\scriptstyle{{{{0\!\!\atop{}}{1\!\!\atop{}}{2\atop1}\!\!{1\!\!\atop{}}{1\!\!\atop{}}{0\!\!\atop{}}{0\!\!\atop{}}}}, {{{1\!\!\atop{}}{1\!\!\atop{}}{1\atop1}\!\!{1\!\!\atop{}}{1\!\!\atop{}}{1\!\!\atop{}}{1\!\!\atop{}}}}}}\,=\,
N_{\scriptstyle{{{{0\!\!\atop{}}{0\!\!\atop{}}{0\atop0}\!\!{1\!\!\atop{}}{1\!\!\atop{}}{1\!\!\atop{}}{1\!\!\atop{}}}}, {{{1\!\!\atop{}}{2\!\!\atop{}}{3\atop1}\!\!{2\!\!\atop{}}{1\!\!\atop{}}{0\!\!\atop{}}{0\!\!\atop{}}}}}}\,=\,-1,$$
$$N_{{\scriptstyle{{{{0\!\!\atop{}}{0\!\!\atop{}}{1\atop1}\!\!{1\!\!\atop{}}{1\!\!\atop{}}{1\!\!\atop{}}{0\!\!\atop{}}}}, {{{1\!\!\atop{}}{2\!\!\atop{}}{2\atop1}\!\!{2\!\!\atop{}}{1\!\!\atop{}}{0\!\!\atop{}}{0\!\!\atop{}}}}}}}\,=\,
N_{{\scriptstyle{{{{0\!\!\atop{}}{1\!\!\atop{}}{1\atop1}\!\!{1\!\!\atop{}}{1\!\!\atop{}}{0\!\!\atop{}}{0\!\!\atop{}}}}, {{{1\!\!\atop{}}{1\!\!\atop{}}{2\atop1}\!\!{1\!\!\atop{}}{1\!\!\atop{}}{1\!\!\atop{}}{1\!\!\atop{}}}}}}}\,=\,
N_{{\scriptstyle{{{{0\!\!\atop{}}{0\!\!\atop{}}{1\atop0}\!\!{1\!\!\atop{}}{1\!\!\atop{}}{1\!\!\atop{}}{1\!\!\atop{}}}}, {{{1\!\!\atop{}}{2\!\!\atop{}}{2\atop1}\!\!{2\!\!\atop{}}{1\!\!\atop{}}{0\!\!\atop{}}{0\!\!\atop{}}}}}}}\,=\,
N_{{\scriptstyle{{{{0\!\!\atop{}}{1\!\!\atop{}}{1\atop1}\!\!{1\!\!\atop{}}{1\!\!\atop{}}{0\!\!\atop{}}{0\!\!\atop{}}}}, {{{1\!\!\atop{}}{1\!\!\atop{}}{2\atop1}\!\!{2\!\!\atop{}}{1\!\!\atop{}}{1\!\!\atop{}}{0\!\!\atop{}}}}}}}\,=\,1.$$
Since $N_{\alpha,\beta}=-N_{-\alpha,-\beta}$ by \cite{GS},  a straightforward computation shows that
\begin{eqnarray*}
v&=&2f_{{{0\!\!\atop{}}{0\!\!\atop{}}{1\atop0}\!\!{1\!\!\atop{}}{1\!\!\atop{}}{1\!\!\atop{}}{1\!\!\atop{}}}}+
2f_{{{0\!\!\atop{}}{1\!\!\atop{}}{1\atop0}\!\!{1\!\!\atop{}}{1\!\!\atop{}}{1\!\!\atop{}}{0\!\!\atop{}}}}
-f_{{{0\!\!\atop{}}{0\!\!\atop{}}{1\atop1}\!\!{1\!\!\atop{}}{1\!\!\atop{}}{0\!\!\atop{}}{0\!\!\atop{}}}}
-2f_{{{0\!\!\atop{}}{1\!\!\atop{}}{1\atop1}\!\!{1\!\!\atop{}}{1\!\!\atop{}}{0\!\!\atop{}}{0\!\!\atop{}}}}+\\	
&+&f_{{{0\!\!\atop{}}{1\!\!\atop{}}{1\atop1}\!\!{1\!\!\atop{}}{1\!\!\atop{}}{0\!\!\atop{}}{0\!\!\atop{}}}}
-3f_{{{0\!\!\atop{}}{1\!\!\atop{}}{2\atop1}\!\!{1\!\!\atop{}}{1\!\!\atop{}}{0\!\!\atop{}}{\!\!\atop{}}}}-
f_{{{0\!\!\atop{}}{0\!\!\atop{}}{1\atop0}\!\!{1\!\!\atop{}}{1\!\!\atop{}}{1\!\!\atop{}}{1\!\!\atop{}}}}
-f_{{{0\!\!\atop{}}{0\!\!\atop{}}{0\atop0}\!\!{1\!\!\atop{}}{1\!\!\atop{}}{1\!\!\atop{}}{1\!\!\atop{}}}}\\
&=&-f_{{{0\!\!\atop{}}{0\!\!\atop{}}{0\atop0}\!\!{1\!\!\atop{}}{1\!\!\atop{}}{1\!\!\atop{}}{1\!\!\atop{}}}}
-f_{{{0\!\!\atop{}}{0\!\!\atop{}}{1\atop1}\!\!{1\!\!\atop{}}{1\!\!\atop{}}{0\!\!\atop{}}{0\!\!\atop{}}}}
-f_{{{0\!\!\atop{}}{1\!\!\atop{}}{1\atop1}\!\!{1\!\!\atop{}}{1\!\!\atop{}}{0\!\!\atop{}}{0\!\!\atop{}}}}
+2f_{{{0\!\!\atop{}}{1\!\!\atop{}}{1\atop0}\!\!{1\!\!\atop{}}{1\!\!\atop{}}{1\!\!\atop{}}{0\!\!\atop{}}}}
-3f_{{{0\!\!\atop{}}{1\!\!\atop{}}{2\atop1}\!\!{1\!\!\atop{}}{1\!\!\atop{}}{0\!\!\atop{}}{\!\!\atop{}}}}
+f_{{{0\!\!\atop{}}{0\!\!\atop{}}{1\atop0}\!\!{1\!\!\atop{}}{1\!\!\atop{}}{1\!\!\atop{}}{1\!\!\atop{}}}}.
\end{eqnarray*}
It is worth mentioning that $v$ also appears in the extended (unpublished) version of \cite{LT2} as a linear combination of vectors $v_{12}$ and $v_{14}$.

Since $u$ is a highest weight vector of weight $1$ for $\g_e(0)$ it must be that 
$[u,v]\in \g_e(6)^{\g_e(0)}$. By \cite[p.~150]{LT2}, the latter subspace is spanned by
$e_{{{1\!\!\atop{}}{1\!\!\atop{}}{1\atop1}\!\!{0\!\!\atop{}}{0\!\!\atop{}}{0\!\!\atop{}}{0\!\!\atop{}}}}+
e_{{{1\!\!\atop{}}{1\!\!\atop{}}{1\atop0}\!\!{1\!\!\atop{}}{0\!\!\atop{}}{0\!\!\atop{}}{0\!\!\atop{}}}}+2e_{{{0\!\!\atop{}}{1\!\!\atop{}}{2\atop1}\!\!{1\!\!\atop{}}{0\!\!\atop{}}{0\!\!\atop{}}{0\!\!\atop{}}}}$.
On the other hand, a rough calculation (ignoring the signs of structure constants) shows that $[u,v]$ is a linear combination of $e_{{{1\!\!\atop{}}{1\!\!\atop{}}{1\atop1}\!\!{0\!\!\atop{}}{0\!\!\atop{}}{0\!\!\atop{}}{0\!\!\atop{}}}}$ and
$e_{{{1\!\!\atop{}}{1\!\!\atop{}}{1\atop0}\!\!{1\!\!\atop{}}{0\!\!\atop{}}{0\!\!\atop{}}{0\!\!\atop{}}}}$.
This implies that
\begin{equation}
[u,v]=0.
\end{equation}
\subsection{Searching for a quadratic relation in $\mathbf{U(\g,e)}^{\mathbf{\mathrm ab}}$}
\label{ss6}
Similar to our discussion in (\ref{ss2}) we hope (with fingers crossed) that the element $[\Theta_u,\Theta_v]$ lies in
$\F_8(Q)\setminus \F_7(Q)$. For that purpose we have to look closely at the element $\varphi:=\{\theta_u,\theta_v\}\in \P_8(\g,e)$ Here, as before, $\theta_y$ denotes the $\F$-symbol of
$\Theta_y$ in the Poisson algebra $\P(\g,e)=\gr_{\F}(U(\g,e))$.

As in (\ref{ss2}) we identify $\P(\g,e)$ with the symmetric algebra $S(\g_e)$ and write $\mathcal{J}$ for the ideal of $\P(\g,e)$ generated by the graded subspace
$[\g_e(0),\g_e(2)]\oplus\textstyle{\sum}_{i\ne 2}\g_e(i)$. We know from \cite[p.~150]{LT2} that
$[\g_e(0),\g_e(2)]$ is an irreducible $\g_e(0)$-module of highest weight $4$ and 
$\g_e(2)=[\g_e(0),\g_e(2)]\oplus \g_e(2)^{\g_e(0)}$. Furthermore, $\g_e(2)^{\g_e(0)}$ is a $2$-dimensional subspace spanned by $e$ and $e_0:=e_2+e_5+e_7+e_8$. It follows that the factor-algebra $\bar{\P}(\g,e):=\P(\g,e)/\mathcal{J}$ is isomorphic to a 
polynomial algebra in $e$ and $e_0$. We let $\bar{\varphi}$ denote the image of $\varphi$ in $\bar{\P}(\g,e)$. Then
$$\bar{\varphi}=\lambda e^2+\mu ee_0+\nu e_0^2,$$
and the main results of \cite[Theorem~1.2]{PT21} imply that the scalars $\lambda,\mu,\nu$ lie in
the ring $R$.
Since is immediate from \cite[Prop.~2.1]{Pr14} and \cite[5.2]{PT21} that the largest commutative quotient of $U(\g,e)$ is generated by the image of $\Theta_e$ we would find a desired quadratic relation in $U(\g,e)^{\rm ab}$  if we managed to prove that the coefficient $\lambda$
of $\bar{\varphi}$ is nonzero. Indeed, let $I_c$ denote the $2$-sided ideal of $U(\g,e)$ generated by all commutators. If $\lambda\in R^\times$ then the element  $[\Theta_u,\Theta_v]\in I_c\cap Q_R$ has Kazhdan degree $8$ and is congruent to $\lambda \Theta_e^2$ modulo $I_c\cap U(\g_R,e)+ \F_7(Q_R)$. As it follows from \cite[Prop.~5.4]{PT21} that $$U(\g,e)\cap \F_7(Q_R)\subset R 1+R\Theta_e+I_c\cap U(\g_R,e)$$ the latter would imply that 
$\lambda\Theta_e^2+\eta\Theta_e+\xi 1\in I_c$ for some $\eta,\xi\in R$.
\begin{Lemma}\label{l1} We have that
	$[\g_e(1),\g_e(1)]^{\g_e(0)}=\C e_0$.
\end{Lemma} \begin{proof} It follows from \cite[4.4]{PS18} that $[\g_e,\g_e](2)=[\g_e(0),\g_e(2)]+[\g_e(1),\g_e(1)]$ has codimension $1$ in $\g_e(2)$. Hence
$[\g_e,\g_e](2)^{\g_e(0)}\ne \{0\}$. On the other hand,
 \cite[p.~150]{LT2} shows that $[\g_e(0),\g_e(2)]\cong L(4)$ and $[\g_e(1),\g_e(1)]$ is a homomorphic image of $\wedge^2 L(3)$, where $L(r)$ stands for the irreducible $\sl(2)$-module of highest weight $r$. This implies that the subspace $[\g_e(1),\g_e(1)]^{\g_e(0)}$
is $1$-dimensional. 

By \cite[p.~150]{LT2}, the $\g_e(0)$-module $\g_e(1)$ is generated by the highest weight vector
$$w:=e_{{{2\!\!\atop{}}{3\!\!\atop{}}{4\atop2}\!\!{4\!\!\atop{}}{3\!\!\atop{}}{2\!\!\atop{}}{1\!\!\atop{}}}}
-e_{{{1\!\!\atop{}}{3\!\!\atop{}}{5\atop3}\!\!{4\!\!\atop{}}{3\!\!\atop{}}{2\!\!\atop{}}{1\!\!\atop{}}}}.$$
Given a root $\gamma\in\Phi$ we write $\nu_3(\gamma)$ for the coefficient of $\alpha_3$ in the expression of $\gamma$ as a linear combination of the simple roots $\alpha_i\in\Pi$, and we denote by $t_3$ the derivation of $\g$ such that $t_3(e_\gamma)=\nu_3(\gamma)e_\gamma$ for all $\gamma\in\Phi$.
Then $t_3(w)=3w$ and $t_3(f')=-2f'$. Our preceding remarks show that the subspace $[\g_e(1),\g_e(1)]^{\g_e(0)}$ is spanned by a nonzero vector of the form 
$$a[(\ad\,f')^3(w),w]+b [(\ad\,f')^2(w),(\ad f')(w)]$$ with $a,b\in\C$. Since such a vector is a linear combination of $e$ and $e_0$ and lies in the kernel of $t_3$ we now deduce that
$[\g_e(1),\g_e(1)]^{\g_e(0)}=\C e_0$ as stated (one should keep in mind here that $t_3(e_0)=0$ and $t_3(e)=e_3\ne 0$).
\end{proof}
Let $h_0=[e_0,f]=h_2+h_5+2h_7+2h_8$. Since $[e,e_0]=0$ we have that $[h_0,e]=[[e_0,f],e]=[h,e_0]=2e_0.$
Since $h_0\in\t$, each $e_i$ is an eigenvector for $\ad h_0$ this forces $[h_0,e_0]=2e_0$. Next we set $f_0:=\textstyle{\frac{1}{2}}[f,[f,e_0]]=\textstyle{\frac{1}{2}}[h_0,f]
$  and observe that
$$[e,f_0]=\textstyle{\frac{1}{2}}\big([h,[f_,e_0]]+[f,[h,e_0]]\big)=[f,e_0]=-h_0.$$
Since $[f,e_0]=-h_0$ we get $[f_0,e_0]=\frac{1}{2}[[h_0,f],e_0]=[e_0,f]=-[e,f_0]$ which yields $$[f_0,[f_0,e]]=-[f_0,[f,e_0]]=-[f,[f_0,e_0]]=[f,[e,f_0]]=-[h,f_0]=2f_0.$$
As both $f$ and $f_0$ lie in $\g_f(-2)^{\g_e(0)}$, they are orthogonal to $[\g_e(0),\g_e(2)]$ with respect to our symmetric bilinear form $(\,\cdot\,,\,\cdot\,)$. 
Since $$(e_0,f_0)=(e_0,\textstyle{\frac{1}{2}}[h_0,f])=\textstyle{\frac{1}{2}}([e_0,h_0],f)=-(e_0,f)=-(1+1+2+2)=-6$$ we have that $(e_0,f+f_0)=0$. As
$(e,f_0) = \textstyle{\frac{1}{2}}(e,[f,[f,e_0]])=\textstyle{\frac{1}{2}}(h,[f,e_0]])=-(f,e_0)=-6$ we get $(e,f+f_0)=(e,f)-(e_0,f_0)=34-6=28$. Since the ideal $\mathcal{J}$ vanishes on 
$\g_f(-2)^{\g_e(0)}$ it follows that
\begin{equation}
	\label{varphi1}
	\varphi(f+f_0)=\bar{\varphi}(f+f_0)=\lambda(e,f+f_0)^2=2^47^2\lambda.\end{equation}  As in (\ref{varphi}) this indicates that we might expect some complications in characteristic $7$. 
\subsection{Computing $\mathbf{\lambda}$}
In order to determine $\lambda$ we use the method described in Subsections~\ref{ss3} and \ref{ss4}. We adopt the notation introduced there and put $E:=\ad e$, $E_0:=\ad e_0$,
$H:=\ad h$, $H_0:=\ad h_0$, $F=\ad f$ and $F_0:=\ad f_0$. Since $u$ and $v$ are in $\g_e(3)$ there exist
$u_-\in \C F^3(u)$ and $v_-\in \C F^3(v)$ such that $u=E^3(u_-)$ and $v=E^3(v_-)$. As $\g_e\cap \g(-5)=\{0\}$ it follows from the $\sl_2$-theory that the elements $u_-$ and $v_-$ lie in $\g_f(-3)$. Arguing as in Subsection~\ref{ss3} we observe that
$$\{\theta_u,\theta_v\}\,=\,\sum_{i=1}^{2s}[u,z_i^*][v,z_i]
+q(u,v)+\mbox{terms of standard degree} \ge 3.$$
Since
all terms of standard degree $\ge 3$ involved in $\{\theta_u,\theta_v\}$ have Kazhdan degree $8$  they must vanish at $f+f_0\in \g(-2)$. Since each quadratic monomial involved in $q(u,v)$ has a linear factor of standard degree $\ge 3$ we also have that $q(u,v)(f+f_0)=0$.
Using the $\g$-invariance of $(\,\cdot\,,\,\cdot\,)$ and the fact that $E^3(f+f_0)=0$  we get
$\{\theta_u,\theta_v\}(f+f_0)=$
\begin{eqnarray*}
&=&\sum_{i=1}^{2s}([u,z_i^*],f+f_0)([v,z_i], f+f_0)\,=\,
\sum_{i=1}^{2s}([E^3(u_-),z_i^*],f+f_0)([E^3(v_-),z_i],f+f_0)\\
&=&\sum_{i=1}^{2s}(z_i^*,[E^3(u_-),f+f_0])([z_i,E^3(v_-)f+f_0])\\
&=&\!\!\sum_{i=1}^{2s}(z_i^*,E^3([u_-,f+f_0])-
3E([E(u_-),h-h_0]))(z_i,E^3([v_-,f+f_0])-3E([E(v_-),h-h_0]))\\
&=&\!\!\sum_{i=1}^{2s}(e,[E^2([u_-,f+f_0])-3[E(u_-),h-h_0],z_i^*])(e,[E^2([v_-,f-f_1])-
3[E(v_-),h-h_0],z_i]).
\end{eqnarray*}
Here we used the fact that $E(f+f_0)=h+[e,f_0]=h-[e_0,f]=h-h_0$. 
As before, our choice of the $z_i^*$'s implies that $\langle x,y\rangle=\sum_{i=1}^{2s}\langle z_i^*,x\rangle\langle z_i,y\rangle$ for all $x,y\in \g(-1)$.
The definition of $\langle\,\cdot\,,\cdot\,\rangle$ then yields
that $\{\theta_u,\theta_v\}(f+f_0)=$
\begin{eqnarray}
&=&\big(e,\big[[E^2([u_-,f+f_0])-3[E(u_-),h-h_0],[E^2([v_-,f+f_0])-3[E(v_-),h-h_0]\big]\big)\\
&=&\big([E^3(u_-),f+f_0],E^2([v_-,f+f_0])-3[E(v_-),h-h_0]\big).
\end{eqnarray}
One should keep in mind here that $E^3([u_-,f+f_0])-3[e,[E(u_-),E(f-f_0)]]=[E^3(u_-),f+f_0]$ which holds since $E^3(f+f_0)=0$. As $E^4(u_-)=0$ the latter equals to
\begin{equation}\label{4.5}([u,E^2(f_0)],[v_-,f_0])-3(u,[f+f_0,[E(v_-),h-h_0]])\end{equation} thanks to the $\g$-invariance of $(\,\cdot\,,\,\cdot\,)$. Recall that $-h_0=[e,f_0]=[f,e_0]$ and $[h,f_0]=-2f_0$. Also, $[f,v_-]=0$ and $[f,f_0]=0$.
By the Jacobi identity, (\ref{4.5}) equals to 
\begin{eqnarray*}
&&([u,E^2(f_0)],[v_-,f_0])-3(u,[[[f+f_0,e],v_-],h-h_0])-3(u,[[e,[f+f_0,v_-]],h-h_0])\\
&-&3(u,[[e,v_-],2f+[f,[e,f_0]]]+2f_0+[f_0,[e,f_0]]])\\
&=&([u,[e,[f,e_0]]],[v_-,f_0])-3\big(u,(H+[E,F_0])^2(v_-)\big)-3([u,[e,[f_0,v_-]],h+[e,f_0])\\
&-&3(u,[E(v_-),2f+2f_0+2f_0-{F_0}^2(e)]).
\end{eqnarray*}
Since $h$ commutes with $[u,[e,[f_0,v_-]]$ the last expression equals
\begin{eqnarray*}
&&2([u,e_0],[v_-,f_0])-3\big(u,(H+[E,F_0])^2(v_-)\big)+3([u,[e,f_0]],[e,[f_0,v_-]])\\
&-&3(u,[E(v_-),2f+2f_0+2f_0-2f_0])\,=\,-2([[u,e_0],f_0],v_-)\\
&-&3\big(u,(H+[E,F_0])^2(v_-)\big)-3([u,E^2(f_0)],[f_0,v_-])+6([u,v_-],h+[e,f_0])\\
&=&\!\!\!-2([[u,e_0],f_0],v_-)-3\big((H-[E_0,F])^2(u),v_-)\big)-6([u,e_0],[f_0,v_-])+6((H-[E_0,F])(u),v_-)\\
&=&\!\!\!-8([f_0,[e_0,u]],v_-)-3\big((3-[E_0,F])^2(u),v_-\big)+6\big((3-[E_0,F])(u),v_-\big)\\
&=&\!\!\!-3\big(([E_0,F]^2-4[E_0,F]+3)(u),v_-\big)-8\big([f_0,[e_0,u]],v_-\big).
\end{eqnarray*}	
Finally, a GAP computation\footnote{Again, see \url{https://github.com/davistem/the_number_of_multiplicity-free_primitive_ideals/} for the code.} reveals that
$$-3\big(([E_0,F]^2-4[E_0,F]+3)(u),v_-\big)-8\big([f_0,[e_0,u]],v_-\big)\,=\,1176\,=\,2^3\cdot 3\cdot 7^2.$$ In view of (\ref{varphi1}) the factor $7^2$ gets cancelled and we obtain $\lambda=\frac{3}{2}\in R^\times$. Arguing as in Subsection~\ref{ss2} we now deduce that $U(\g,e)^{\rm ab}$ has dimension $2$.
\begin{Remark}
For safety, we have also used GAP to compute the expressions (4.3) and (4.4), and the number $1176$ was the output in both cases.	
\end{Remark}
\subsection{The modular case} In this subsection we prove Theorem~B. 
First suppose that $e$ has Bala--Carter label ${\rm A}_5+{\rm A}_1$. By \cite[3.16]{Pr14}, we then have 
\begin{eqnarray*}
\Lambda+\rho&=&\textstyle{\frac{1}{3}}\varpi_1+\textstyle{\frac{1}{3}}\varpi_2+
\textstyle{\frac{1}{6}}\varpi_3+ \textstyle{\frac{1}{6}}\varpi_4
+\textstyle{\frac{1}{6}}\varpi_5+\textstyle{\frac{1}{6}}\varpi_6+
\textstyle{\frac{1}{6}}\varpi_7+\textstyle{\frac{1}{6}}\varpi_8,\\ 
\Lambda'+\rho&=& \textstyle{\frac{1}{3}}\varpi_1+\textstyle{\frac{1}{3}}\varpi_2+\textstyle{\frac{1}{6}}\varpi_3+ \textstyle{\frac{7}{6}}\varpi_4
-\textstyle{\frac{11}{6}}\varpi_5+\textstyle{\frac{7}{6}}\varpi_6+\textstyle{\frac{1}{6}}\varpi_7+
\textstyle{\frac{1}{6}}\varpi_8.
\end{eqnarray*}
In view of \cite[Planche~VII]{B}, we get $\Lambda+\rho\,=\,\textstyle{\frac{1}{6}}(\varpi_1+\varpi_2)+\textstyle{\frac{1}{6}}\rho\,=\,$
$$
\,=\,\textstyle{\frac{1}{3}}\varepsilon_8+\frac{1}{12}(\varepsilon_1+\varepsilon_2+\varepsilon_3+
\varepsilon_4+
\varepsilon_5+\varepsilon_6+\varepsilon_7+5\varepsilon_8) 
+\textstyle{\frac{1}{6}}(\varepsilon_2+2\varepsilon_3+3\varepsilon_4+4\varepsilon_5+
5\varepsilon_6+6\varepsilon_7+23\varepsilon_8).
$$
Using the standard coordinates of $\mathbb{R}^8$ we obtain
$\Lambda+\rho=\frac{1}{12}(1,3,5,7,9,11,13,55)$.

Next we observe that $\Lambda'+\rho\,=\,\Lambda+\rho+\varpi_4-2\varpi_5+\varpi_6$. Since
$\varpi_4-2\varpi_5+\varpi_6\,=\,$
$$\,=\,(0,0,1,1,1,1,1,5)-2(0,0,0,1,1,1,1,4)+(0,0,0,0,1,1,1,3)
\,=\,(0,0,1,-1,0,0,0)$$ by \cite[Planche~VII]{B}, we have $\Lambda'+\rho=\frac{1}{12}(1,3,17,-5,9,11,13,55)$. It follows that
$$(\Lambda+\rho\,\vert\, \Lambda+\rho)-(\Lambda'+\rho\,\vert\, \Lambda'+\rho)\,=\,
\textstyle{\frac{1}{144}}\big((5^2-17^2)+(7^2-5^2)\big)\,=\,
\textstyle{\frac{1}{144}}(7^2-17^2))\,=\,-\textstyle{\frac{5}{3}}.$$

Now suppose that $e$ has Bala--Carter label ${\rm D}_5({\rm a}_1)+{\rm A}_2$. By \cite[3.17]{Pr14}, 
\begin{eqnarray*}
	\Lambda+\rho&=&\textstyle-{\frac{1}{4}}\varpi_1-\textstyle{\frac{1}{4}}\varpi_2-
	\textstyle{\frac{1}{4}}\varpi_3+ \varpi_4
	-\textstyle{\frac{1}{4}}\varpi_5+\varpi_6+
	-\textstyle{\frac{1}{4}}\varpi_7-\textstyle{\frac{1}{4}}\varpi_8,\\ 
	\Lambda'+\rho&=& -\textstyle{\frac{1}{4}}\varpi_1-\textstyle{\frac{1}{4}}\varpi_2-\textstyle{\frac{1}{4}}\varpi_3+ 2\varpi_4
	-\textstyle{\frac{9}{4}}\varpi_5+2\varpi_6-\textstyle{\frac{1}{4}}\varpi_7-
	\textstyle{\frac{1}{4}}\varpi_8.
\end{eqnarray*}
Hence $\Lambda+\rho=-\frac{1}{4}\rho+\frac{5}{4}(\varpi_4+\varpi_6)\,=\,$
$$
\,=\,-\textstyle{\frac{1}{4}}(0,1,2,3,4,5,6,23)+\textstyle{\frac{5}{4}}(0,0,1,1,2,2,2,8)\,=\,
\textstyle{\frac{1}{4}}(0,-1,3,2,6,5,4,17).$$ Similarly,  $\Lambda'+\rho=-\frac{1}{4}\rho+\frac{9}{4}(\varpi_4+\varpi_6)-2\varpi_5\,=\,$
$$
=-\textstyle{\frac{1}{4}}(0,1,2,3,4,5,6,23)+\textstyle{\frac{9}{4}}(0,0,1,1,2,2,2,8)-
(0,0,0,0,2,2,2,8)=
\textstyle{\frac{1}{4}}(0,-1,7,6,-3,-4,-5,17).$$
Therefore,
$$(\Lambda+\rho\,\vert\, \Lambda+\rho)-(\Lambda'+\rho\,\vert\, \Lambda'+\rho)\,=\,
\textstyle{\frac{1}{16}}(2^2-7^2)\,=\,-\textstyle{\frac{45}{16}}.$$ This shows that in both cases
the element $(\Lambda+\rho\,\vert\, \Lambda+\rho)-(\Lambda'+\rho\,\vert\, \Lambda'+\rho)$ is  invertible in $R$. We set $r:=(\Lambda'+\rho\,\vert\, \Lambda'+\rho)-(\rho\,\vert\,\rho)$ and 
$r':=(\Lambda'+\rho\,\vert\, \Lambda'+\rho)-(\rho\,\vert\,\rho)$.  Clearly, $r,r'\in R$.

Since the ideals $I(\Lambda)$ and $I(\Lambda')$ are multiplicity-free, our discussion in the introduction shows that $I(\Lambda)={\rm Ann}_{\,U(\g)}\,V$ and 
$I(\Lambda')={\rm Ann}_{\,U(\g)\,}V'$
for some $1$-dimensional $U(\g,e)$-modules $V$ and $V'$. There exist $2$-sided ideals $I$ and $I'$ of codimension $1$ in $U(\g,e)$ such that $V=U(\g,e)/I$ and $V'=U(\g,e)/I'$. As $L(\Lambda)$ and $L(\Lambda')$ are highest weight modules, we can find a Casimir element $C\in U(\g_R)$ which acts on $L(\Lambda)$  and $L(\Lambda')$ as $r{\rm Id}$ and
$r'{\rm Id}$, respectively.

Obviously, $C-r\in I$, $C-r'\in I'$, and the ideals $I$ and $I'$ contain all commutators in $U(\g,e)$. Put $I_R:=I\cap U(\g_R,e)$, $I_R':=I'\cap U(\g_R,e)$ and $V_R:=U(\g_R,e)/I_R$, $V'_R:=U(\g_R,e)/I_R'$. It follows from \cite[Proposition~5.4]{PT21} that $U(\g_R,e)=R\,1\oplus I_R$ and $U(\g_R,e)=R\,1\oplus I_R'$. 

To ease notation we identify $e$ with its image in $\g_\k=\g_R\otimes_R\k$ (this will cause no confusion).  Following \cite{GT18} we let $U(\g_\k,e)$ denote the modular finite $W$-algebras assiciated with the pair $(\g_\k,e)$. By \cite[Theorem~1.2(1)]{PT21},
we have that $U(\g_\k,e)\cong U(\g_R,e)\otimes_R\k$ as $\k$-algebras. Our computations in 
Subsections~\ref{ss3} and \ref{ss4} imply that the image of $C$ in the largest commutative quotient
of $U(\g_\k,e)$ satisfies a non-trivial quadratic equation. As a consequence, $U(\g_\k,e)$ cannot have more 
than two $1$-dimensional representations. On the other hand, the formulae for $r-r'$ obtained earlier yield that in each case the image of $r-r'$ in $R/pR\subset\k$ is nonzero for any good prime $p$ of $G_\Z$. This entails that
$V_\k:=V_R\otimes_R\k$ and $V_\k':=V_R'\otimes_R\k$ are the only non-equivalent $1$-dimensional representations of $U(\g_\k,e)$.

Given $\xi\in\g_\k^*$ we let $\g_\k^\xi$ denote the coadjoint stabiliser of $\xi$ in $\g_\k$.
As explained in \cite[8.1]{GT18} the modular finite $W$-algebra $U(\g_\k,e)$ contains a large 
central subalgebra $Z_p(\g_\k,e)$ isomorphic to a polynimial algebra in 
$\dim \g_\k^\chi$ variables. The algebra $U(\g_\k,e)$ is free $Z_p(\g_\k,e)$-module of rank $p^{\dim\g_\k^\chi}$ and the maximal spectrum of $Z_p(\g_\k,e)$ identifies with
a Frobernius twist of a {\it good transverse slice} $\mathbb{S}_\chi=\chi+\tilde{\kappa}(\mathfrak{o})$ to the coadjoint orbit of $\chi$. Here $\tilde{\kappa}\colon\,\g_\k\to \g_\k^*$ is the $G_\k$-module isomorphism induced by the Killing form $\kappa$ and $\mathfrak{o}$ is a graded subspace of 
$\bigoplus_{i\le 0}\,\g_\k(i)$  complementary to the tangent space $T_e((\Ad G_\k)\,e)=[e,\g_\k]$.

Every $\xi\in \mathbb{S}_\chi$ gives rise to a maximal ideal $J_\xi$ of $Z_p(\g_\k,e)$ which leads to a {\it $p$-central reduction} $$U_\xi(\g_\k,e)\,:=\,U(\g_\k,e)/J_\xi\, U(\g_\k,e)\,\cong\,
U(\g_\k,e)\otimes_{\,Z_p(\g_\k,e)}\,\k_\xi.$$ By \cite[Lemma~2.2(iii)]{Pr10} and
\cite[Sections~8 and 9]{GT18}, for every $\xi\in\mathbb{S}_\chi$ we have an algebra isomorphism
\begin{equation}\label{mat}
U_\xi(\g_\k)\,\cong\,{\rm Mat}_{p^{d(\chi)}}(U_\xi(\g_\k,e)).
\end{equation}
The $1$-dimensional $U(\g_\k,e)$-modules $V_\k$ and $V_\k'$ are annihilated by some maximal ideals
$J_\eta$ and $J_{\eta'}$ of $Z_p(\g_\k,e)$. Therefore, $V_\k$and $V_\k'$ are $1$-dimensional modules over the $p$-central reductions $U_\eta(\g_k,e)$ and $U_{\eta'}(\g_\k,e)$, respectively.  
By (\ref{mat}), the reduced enveloping algebras
 $U_\eta(\g_\k)$ and $U_{\eta'}(\g_\k)$ with $\eta,\eta'\in \mathbb{S}_\chi$ afford simple modules of dimension $p^{d(\chi)}$; we call them $\widetilde{V}_\k$ and $\widetilde{V}_\k'$. As explained in \cite[Lemma~2.2(iii)]{Pr10} and \cite[Sections~8 and 9]{GT18} we may assume further that the $U(\g_\k)$-modules $\widetilde{V}_\k$ and $\widetilde{V}_\k'$ are generated by their $1$-dimensional subspaces $V_\k$ and $V_\k'$, respectively.
 
 At this point we invoke a contracting $\k^\times$-action on $\mathbb{S}_\chi$ given by 
 $\mu(t)\cdot \xi=t^{-2}(\Ad^* \tau(t))\,\xi$ for all $t\in\k^\times$ and $\xi\in\mathbb{S}_\chi$.
 It shows, in particular, that $\dim (\Ad G_\k)\,\xi\ge \dim (\Ad G_\k)\,\chi$ for every $\xi\in \mathbb{S}_\chi$. In conjunction with the main result of \cite{Pr95} this entails that 
 $\dim (\Ad G_\k)\,\eta=\dim (\Ad G_\k)\,\eta'=\dim (\Ad G_\k)\,\chi$. 
 By \cite[Theorem~3.8]{PS18}, the $G_\k$-orbit of $e$ is rigid in $\g_\k$.
 Therefore, $\chi$ lies in a single sheet of $\g_\k^*$ which coincides with the coadjoint orbit of $\chi$.  Since the contracting action of $\mu(\k^\times)$ on $\mathbb{S}_\chi$ now shows that both $\eta$ and $\eta'$ lie in the only sheet of $\g_\k^*$ containing $\chi$, we deduce that 
 $\chi =(\Ad^* g)\,\eta$ and $\chi=(\Ad^* g')\,\eta'$ for some $g,g'\in G_\k$.

Given $\xi\in\g_\k^*$ we denote by $I_\xi$ the $2$-sided ideal of $U(\g_\k)$ generated by all elements $x^p-x^{[p]}-\xi(x)^p$ with $x\in \g_\k$. It is well-known (and easy to check) that for any $y\in G_\k$ the automorphism $\Ad y$ of $U(\g_\k)$ sends $I_\xi$ onto $I_{(\Ad^* y)\,\xi}$
and thus gives rise to an algebra isomorphism between the respective reduced enveloping algebras. 
The image $C_\k$ of our Casimir element $C$ in $U(\g_\k)=U(\g_R)\otimes_R \k$ lies in the Harish-Chandra centre of $U(\g_k)$. Hence $(\Ad y)(C_\k-a)=C_\k-a$ for all $y\in G_\k$ and $a\in \k$.

Let $\tilde{I}$ and $\widetilde{I}'$ denote the annihilators of $\widetilde{V}_\k$ and $\widetilde{V}'_\k$ in $U(\g_\k)$, and write $\bar{r}$ and $\bar{r}'$ for the images of $r$ and $r'$ in $\k$. The above discussion shows that $\tilde{I}$ contains $I_\eta$ and $C_\k-\bar{r}$ whereas $\tilde{I}'$ contains $I_{\eta'}$ and $C_\k-\bar{r}'$. By construction, $\tilde{I}/I_\eta$ and 
$\tilde{I}'/I_{\eta'}$ have codimension $p^{2d(\chi)}$ in $U_\eta(\g_\k)$ and $U_{\eta'}(\g_\k)$, respectively. Hence the $2$-sided ideals $(\Ad g)(\tilde{I})/(\Ad g)(I_\eta)=(\Ad g)(\tilde{I})/
I_\chi$ and $(\Ad g')(\tilde{I}')/(\Ad g)(I_{\eta'})=(\Ad g)(\tilde{I}')/
I_\chi$ have codimension $p^{2d(\chi)}$ in $U_\chi(\g_\k)=U(\g_\k)/I_\chi$. These ideals are distinct since $(\Ad g)(C_\k)=(\Ad g')(C_\k)=C_\k$ and $\bar{r}\ne \bar{r}'$. Thanks to the main result of \cite{Pr95} this yields that 
$U_\chi(\g_\k)$ has at least two simple modules of dimension $p^{d(\chi)}$. 
On the other hand, being a homomorphic image of $U(\g_\k,e)$ the algebra $U_\chi(\g_\k,e)$ cannot have more than two $1$-dimensional representations. Applying (\ref{mat}) with $\xi=\chi$ we finally deduce that $U_\chi(\g_\k)$ has exactly two simple modules of dimension $p^{d(\chi)}$. This completes the proof of Theorem~B.

\end{document}